\RequirePackage{etex}
\pdfoutput=1
\documentclass{amsart}
\usepackage{mathtools}
\usepackage{amssymb}
\usepackage{xparse}
\usepackage{xr-hyper}
\usepackage{array}

\usepackage{subdepth}
\usepackage{bm,bbm}
\usepackage{mathrsfs}
\usepackage{stmaryrd}
\usepackage[scr=boondoxo]{mathalfa}

\usepackage[centering, footskip=6mm]{geometry}

\usepackage{xpatch}
\usepackage{mdframed}

\makeatletter
\xpatchcmd{\endmdframed}
  {\aftergroup\endmdf@trivlist\color@endgroup}
  {\endmdf@trivlist\color@endgroup\@doendpe}
  {}{}
\makeatother

\usepackage{iftex}

\ifpdf
\usepackage[hypertexnames=false]{hyperref}
\else
\usepackage[dvipdfmx,hypertexnames=false]{hyperref}
\fi

\usepackage[nameinlink]{cleveref}

\usepackage{etoolbox}
\makeatletter
\patchcmd\@thm
  {\let\thm@indent\indent}{\let\thm@indent\noindent}%
  {}{}
\expandafter\patchcmd\csname\string\proof\endcsname
  {\normalparindent}{0pt }{}{}
\makeatother

\makeatletter
\renewcommand\th@plain{\slshape}
\makeatother
\newtheoremstyle{plain}
  {-\topsep}
  {}
  {\slshape}
  {}
  {\sffamily\bfseries}
  {.}
  {.5em}
  {}
\theoremstyle{plain}
\newtheorem{theorem}{Theorem}
\newtheorem{pretheorem}{Theorem}

\newtheorem{corollary}[theorem]{Corollary}
\newtheorem{lemma}[theorem]{Lemma}
\newtheorem{proposition}[theorem]{Proposition}

\newtheorem*{claim*}{Claim}

\surroundwithmdframed[skipabove=\medskipamount,skipbelow=\medskipamount]{theorem}
\surroundwithmdframed[skipabove=\medskipamount,skipbelow=\medskipamount]{pretheorem}
\surroundwithmdframed[skipabove=\medskipamount,skipbelow=\medskipamount]{corollary}
\surroundwithmdframed[skipabove=\medskipamount,skipbelow=\medskipamount]{lemma}
\surroundwithmdframed[skipabove=\medskipamount,skipbelow=\medskipamount]{proposition}
\surroundwithmdframed[skipabove=\medskipamount,skipbelow=\medskipamount]{claim}
\surroundwithmdframed[skipabove=\medskipamount,skipbelow=\medskipamount]{claim*}

\newtheoremstyle{definition}
  {-\topsep}
  {}
  {\normalfont}
  {}
  {\sffamily\bfseries}
  {.}
  {.5em}
  {}
\theoremstyle{definition}

\newtheorem{definition}[theorem]{Definition}

\newtheorem{remark}[theorem]{Remark}
\newtheorem{example}[theorem]{Example}
\newtheorem{conjecture}[theorem]{Conjecture}

\surroundwithmdframed[skipabove=\medskipamount,skipbelow=\medskipamount]{exercise}
\surroundwithmdframed[skipabove=\medskipamount,skipbelow=\medskipamount]{definition}
\surroundwithmdframed[skipabove=\medskipamount,skipbelow=\medskipamount]{notation}
\surroundwithmdframed[skipabove=\medskipamount,skipbelow=\medskipamount]{problem}
\surroundwithmdframed[skipabove=\medskipamount,skipbelow=\medskipamount]{question}
\surroundwithmdframed[skipabove=\medskipamount,skipbelow=\medskipamount]{note}
\surroundwithmdframed[skipabove=\medskipamount,skipbelow=\medskipamount]{remark}
\surroundwithmdframed[skipabove=\medskipamount,skipbelow=\medskipamount]{example}
\surroundwithmdframed[skipabove=\medskipamount,skipbelow=\medskipamount]{conjecture}

\crefname{section}{Section}{Sections}

\crefname{theorem}{Theorem}{Theorems}
\crefname{pretheorem}{Theorem}{Theorems}
\crefname{corollary}{Corollary}{Corollaries}
\crefname{lemma}{Lemma}{Lemmas}
\crefname{proposition}{Proposition}{Propositions}
\crefname{claim}{Claim}{Claims}
\crefname{definition}{Definition}{Definitions}
\crefname{notation}{Notation}{Notations}
\crefname{problem}{Problem}{Problems}
\crefname{question}{Question}{Questions}
\crefname{note}{Note}{Notes}
\crefname{remark}{Remark}{Remarks}
\crefname{example}{Example}{Examples}
\crefname{conjecture}{Conjecture}{Conjectures}

\crefname{enumi}{}{}
\crefname{enumii}{}{}
\crefname{enumiii}{}{}

\crefformat{equation}{{\upshape(#2#1#3)}}

\usepackage{autonum}

\makeatletter
\newcommand{\restore@Environment}[1]{%
  \AtBeginDocument{%
    \csletcs{#1*}{#1}%
    \csletcs{end#1*}{end#1}%
  }%
}
\forcsvlist\restore@Environment{alignat,equation,gather,multline,flalign,align}
\makeatother

\usepackage{enumitem}
\setlist{leftmargin=20pt}
\setlist[enumerate]{label=\textup{(\roman*)}}

\ifpdf
\usepackage{graphicx}
\else
\usepackage[dvipdfmx]{graphicx}
\fi
\usepackage{tikz}
\usetikzlibrary{arrows.meta}
\usetikzlibrary{tikzmark}
\usetikzlibrary{cd}
\tikzcdset{ampersand replacement=\&}
\tikzcdset{
  cells={font=\everymath\expandafter{\the\everymath\displaystyle}},
}

\allowdisplaybreaks

\newif\ifendnotes

\endnotesfalse

\ifendnotes
\usepackage{endnotes}
\fi

\newif\ifshowkeys

\showkeysfalse

\ifshowkeys
\usepackage{showkeys}

\makeatletter
  \SK@def\cref#1{\SK@\SK@@ref{#1}\SK@cref{#1}}%
\makeatother
\fi

\let\tmp\phi
\let\phi\varphi
\let\varphi\tmp
\let\tmp\epsilon
\let\epsilon\varepsilon
\let\varepsilon\tmp

\newcommand{\R}{\mathbb{R}}

\newcommand{\midmid}{\mathrel{}\middle|\mathrel{}}
\newcommand{\semicolon}{\nobreak\mskip2mu\mathpunct{}\nonscript
\mkern-\thinmuskip{;}\mskip6muplus1mu\relax}
\renewcommand{\subset}{\subseteq}

\renewcommand{\mod}[1]{(\mathrm{mod}\ #1)}

\renewcommand{\and}{\quad\text{and}\quad}

\makeatletter
\NewDocumentCommand{\xsideset}{mmme{_^}}{%
\mathop{%
\settowidth{\dimen0}{$\m@th\displaystyle#3$}%
\dimen0=.5\dimen0
\settowidth{\dimen2}{$%
\m@th\displaystyle#3%
\IfValueT{#4}{_{#4}}%
\IfValueT{#5}{^{#5}}%
$}%
\dimen2=.5\dimen2
\advance\dimen2 -\dimen0
\sbox6{\scriptspace\z@$\displaystyle{\vphantom{#3}}#1$}
\sbox8{\scriptspace\z@$\displaystyle{\vphantom{#3}}#2$}
\ifdim\wd6>\dimen2 \kern\dimexpr\wd6-\dimen2\relax\fi
{%
\mathop{\llap{\copy6}{\displaystyle#3}\rlap{\copy8}}\limits%
\IfValueT{#4}{_{#4}}%
\IfValueT{#5}{^{#5}}%
}%
\ifdim\wd8>\dimen2 \kern\dimexpr\wd8-\dimen2\relax\fi
}%
}
\makeatother

\begin{document}

\title[{Some remarks on the $[x/n]$-sequence}]
{Some remarks on the $[x/n]$-sequence}
\author[K. Saito]{Kota Saito}
\author[Y. Suzuki]{Yuta Suzuki}
\author[W. Takeda]{Wataru Takeda}
\author[Y. Yoshida]{Yuuya Yoshida}
\maketitle

\newif\ifarXiv
\arXivfalse

\ifarXiv
\begin{abstract}
After the work of Bordell\`{e}s, Dai, Heyman, Pan and Shparlinki (2018),
several authors studied the sum of the type
\[
\sum_{n\le x}a\biggl(\biggl[\frac{x}{n}\biggr]\biggr),
\]
i.e.\ the average of arithmetic function over the $[\frac{x}{n}]$-sequence.
The image of $[\frac{x}{n}]$ without multiplicity
\[
\mathscr{S}(x)
\coloneqq
\{
[x/n]
\mid
n\in[1,x]\cap\mathbb{Z}
\}
\]
is also studied by some authors after the work of Heyman~(2019).
In this paper, we give three remarks on this topic.
Firstly, we show that the exponent $\frac{1}{3}$
in the result
\[
\sum_{\substack{
n\in\mathscr{S}(x)\\
n\equiv a\ \mod{q}
}}
1
=
\frac{2\sqrt{x}}{q}
+
O\biggl(\biggl(\frac{x}{q}\biggr)^{\frac{1}{3}}\log x\biggr)
\quad\text{for $1\le q\le x^{\frac{1}{4}}(\log x)^{-\frac{3}{2}}$}
\]
of Wu and Yu~(2022) on the distribution of the $[\frac{x}{n}]$-sequence
in arithmetic progression can be replaced by an exponent effected by exponent pairs.
This can be done by a simple application of Dirichlet's hyperbola method.
Secondly, we prove an asymptotic formula for the number of primitive lattice points
with coordinates from the $[\frac{x}{n}]$-sequence.
Srichan~(2023) also claimed a result but his method indeed fails to give an asymptotic formula.
We introduce a simple averaging trick to recover the asymptotic formula
\[
\sum_{\substack{
m,n\in\mathscr{S}(x)\\
(m,n)=1
}}
1
=
\frac{4}{\zeta(2)}x
+
O(x^{\frac{1}{2}+\frac{k+\ell}{2k+2}}),
\]
for any exponent pair $(k,\ell)$,
which recovers and improves Srichan's result.
Thirdly, we introduce the ``multiplicative'' version of the Titchmarsh divisor problem:
\[
\sum_{n\le x}\Lambda(n)a\biggl(\biggl[\frac{x}{n}\biggr]\biggr)
\quad\text{and}\quad
\sum_{p\le x}a\biggl(\biggl[\frac{x}{p}\biggr]\biggr).
\]
When $a(n)$ is a ``small'' arithmetic function,
we can deal with these sums without difficulty as the preceding results.
When $a(n)=\phi(n)$, we can still deal with the sum with the von Mangoldt function
by the method of Zhai~(2020).
However, we find the asymptotic formula
\[
\sum_{p\le x}\phi\biggl(\biggl[\frac{x}{p}\biggr]\biggr)
\overset{\text{?}}{\sim}
\frac{1}{\zeta(2)}x\log\log x
\]
(or its disproof) seems to be out of reach of our current technology.
We just give upper and lower bounds for this problem,
the multiplicative Titchmarsh divisor problem with $\phi(n)$.
\end{abstract}
\else
\begin{abstract}
After the work of Bordell\`{e}s, Dai, Heyman, Pan and Shparlinki (2018) and Heyman (2019), several authors studied the averages of arithmetic functions over the sequence $[x/n]$ and the integers of the form $[x/n]$. In this paper, we give three remarks on this topic. Firstly, we improve the result of Wu and Yu (2022) on the distribution of the integers of the form $[x/n]$ in arithmetic progressions by using a variant of Dirichlet's hyperbola method. Secondly, we prove an asymptotic formula for the number of primitive lattice points with coordinates of the form $[x/n]$, for which we introduce a certain averaging trick.
Thirdly, we study a certain ``multiplicative'' analog of the Titchmarsh divisor problem.  We derive asymptotic formulas for such ``multiplicative'' Titchmarsh divisor problems for ``small'' arithmetic functions and the Euler totient function with the von Mangoldt function.
However, it turns out that the average of the Euler totient function over the $[x/p]$-sequence seems rather difficult and we propose a hypothetical asymptotic formula for this average.
\end{abstract}
\fi

\section{Introduction}
\label{sec:intro}
In the theory of arithmetic functions,
we frequently come across sums of the form
\[
\sum_{n\le x}a(n)\biggl[\frac{x}{n}\biggr]
\]
like when one expands the definition of Dirichlet convolution.
Suppose that you are half asleep in front of such an expression.
Then you might confuse the above sum with a rather funny expression
\begin{equation}
\label{dizzy_sum}
\sum_{n\le x}a\biggl(\biggl[\frac{x}{n}\biggr]\biggr).
\end{equation}
Even though such a joke-like appearance,
after Bordell\`{e}s, Dai, Heyman, Pan and Shparlinki~\cite{BDHPS} drew attention to it,
such sums over $[x/n]$-sequence have attracted some interests of experts recently.
Although the sum \cref{dizzy_sum} is easy to handle (e.g.\ see Theorem~1.2 of \cite{Wu_general})
if $a(n)$ is a ``small'' arithmetic function like those bounded as $a(n)\ll n^{\epsilon}$
and if one gets satisfied with the error term of the size $x^{\frac{1}{2}}$,
there are some more profound problems in the sum \cref{dizzy_sum}.
One of the earliest difficulties in \cref{dizzy_sum}
was the case when $a(n)$ is of the size comparable to $n$
like the Euler totient function case $a(n)=\phi(n)$.
The case $a(n)=\phi(n)$ was a conjecture proposed in \cite{BDHPS}
and finally solved by Zhai~\cite{Zhai,Zhai:corrigendum}
using the Vinogradov--Korobov--Walfisz type application
of the exponential sums.
Also, obtaining the error terms better than $x^{\frac{1}{2}}$
needs more involved application of the techniques of the exponential sums
and the situation seems to be far away from being exhausted.
See, e.g.\ 
\cite{Bordelles,%
LiMa:small,LiMa,LiuWuYang,%
MaSun,MaWu:PrimeWithMultiplicity,%
Stucky}.

As a related topic, some researchers are interested in the sequence
\[
\mathscr{S}(x)
\coloneqq
\left\{
\biggl[\frac{x}{n}\biggr]
\midmid
n\in[1,x]\cap\mathbb{Z}
\right\}.
\]
Heyman~\cite{Heyman:Cardinality} proved that this sequence has the cardinality
\[
\#\mathscr{S}(x)=2\sqrt{x}+O(1).
\]
Heyman~\cite{Heyman:Prime} also proved an asymptotic formula
for the number of primes in $\mathscr{S}(x)$
and Ma and Wu~\cite{MaWu} improved this result.
In \cite{Heyman:Cardinality},
the distribution of the elements in $\mathscr{S}(x)$
divisible by a given modulus $d$ was conjectured
and this conjecture was solved by Wu and Yu~\cite{WuYu}
by giving the distribution of $\mathscr{S}(x)$ over arithmetic progressions.
It is interesting that
even thought the sequence $\mathscr{S}(x)$
has a sparse cardinality $x^{\frac{1}{2}}$,
we can still count the number of primes in it.

In this paper, we give three remarks to the study of the above problems.

\subsection{The distribution of \texorpdfstring{$\mathscr{S}(x)$}{S(x)} in arithmetic progressions}
\label{subsec:Sx_AP}
We first improve the result of Wu--Yu~\cite{WuYu}
on the distribution of $\mathscr{S}(x)$ over arithmetic progressions.
For $a,q\in\mathbb{Z}$ with $q\ge1$, let
\[
S^{\ast}(x,q,a)
\coloneqq
\sum_{\substack{
n\in\mathscr{S}(x)\\
n\equiv a\ \mod{q}
}}
1.
\]

\begin{pretheorem}[{Wu and Yu~\cite[Theorem~1.1]{WuYu}}]
\label{prethm:WuYu}
For $x\ge3$ and $1\le q\le x^{\frac{1}{4}}(\log x)^{-\frac{3}{2}}$, we have
\[
S^{\ast}(x,q,a)
=
\frac{2\sqrt{x}}{q}+O\biggl(\biggl(\frac{x}{q}\biggr)^{\frac{1}{3}}\log x\biggr),
\]
where the implicit constant is absolute.
\end{pretheorem}

At the last displayed formula in p.~423 of \cite{WuYu},
Wu and Yu have a error term $D^{2}(x/q)^{-1}$,
which is unavoidable since it comes from the Kuzmin--Landau inequality for the low-frequency phase function.
Also, they have the range $D\asymp (x/q)^{\frac{2}{3}}$ at worst to bound $S_{2,2}^{(\delta)}$.
These observation gives the error term in \cref{prethm:WuYu} seems not to be improvable.
However, we show an appropriate application of Dirichlet's hyperbola method
enables some further improvement.
We state the result also with the counting function
\[
S(x,q,a)
\coloneqq
\sum_{\substack{
n\le x\\\relax
[x/n]\equiv a\ \mod{q}
}}1
\]
counted with the multiplicity.

\begin{theorem}
\label{thm:WuYu_improved}
For an exponent pair $(k,\ell)$, $x\ge1$, $a\in\mathbb{Z}$ and $q\in\mathbb{N}$, we have
\begin{align}
S(x,q,a)
&=
x\sum_{\substack{n=1\\n\equiv a\ \mod{q}}}^{\infty}\frac{1}{n(n+1)}
+
O\biggl(
\biggl(\frac{x}{q}\biggr)^{\frac{k+\ell}{2k+2}}
+
1
\biggr),\\
S^{\ast}(x,q,a)
&=
\frac{2\sqrt{x}}{q}
+
O\biggl(
\biggl(\frac{x}{q}\biggr)^{\frac{k+\ell}{2k+2}}
+
1
\biggr),
\end{align}
where the implicit constants depend only on $(k,\ell)$.
\end{theorem}

Therefore, suitable improvements on the exponent pair
immediately give the improvements over \cref{prethm:WuYu}.
For example, the usual van der Corput A,B processes give a sequence of exponent pairs
$(k,\ell)$ such that the exponents in \cref{thm:WuYu_improved} approaches to
\[
\frac{k+\ell}{2k+2}
\to
R-\tfrac{1}{2}
=
0.3290213568\cdots
\]
as shown in Section~5.4 of \cite{GrahamKolesnik}.

We shall prove \cref{thm:WuYu_improved} in \cref{sec:distribution_AP}.

\subsection{Primitive lattice points with \texorpdfstring{$[x/n]$}{[x/n]} coordinates}
\label{subsec:primitive_lattice_point}
It is a well-known classical result
that the number of primitive lattice points in $\mathbb{Z}^{2}$ satisfies
\[
\sum_{\substack{
1\le m,n\le x\\
(m,n)=1
}}
1
=
\frac{1}{\zeta(2)}x^{2}
+
O(x\log x).
\]
Srichan~\cite[Formula (8), Theorem~3, p.~49]{Srichan}
considered its $\mathscr{S}(x)$-analog and claimed that
\begin{equation}
\label{Srichan:failed_result}
\sum_{\substack{
m,n\in\mathscr{S}(x)\\
(m,n)=1
}}
1
=
\frac{4}{\zeta(2)}x
+
O(x^{\frac{5}{6}}\log x).
\end{equation}
However, the authors could not follow Srichan's argument to bound the sum
\begin{equation}
\label{Srichan:failed_sum}
\sum_{x^{\frac{1}{4}}(\log x)^{-\frac{3}{2}}<d\le x}
\mu(d)
\sum_{\substack{a\in\mathscr{S}(x)\\d\mid a}}
\sum_{\substack{b\in\mathscr{S}(x)\\d\mid b}}
1
\end{equation}
on l.~$-3$ in p.~56. He seems to use the bound
\begin{equation}
\label{Srichan:wrong_bound}
\sum_{\substack{a\in\mathscr{S}(x)\\d\mid a}}1
\ll
\frac{\sqrt{x}}{d},
\end{equation}
but this bound does not hold for large $d$ as we can see by the following example:
\begin{example}
\label{example:}
Let
\[
N\in\mathbb{N},\quad
x=N^{3}
\and
d=N^{2}.
\]
Since $N\in\mathbb{N}$ and $1\le N\le x$, we then have
\[
\mathscr{S}(x)
\owns
\biggl[\frac{x}{N}\biggr]
=
[N^{2}]
=
N^{2}
=
d,
\]
which is of course divisible by $d$. This gives
\begin{equation}
\label{ex:below_1}
\sum_{\substack{a\in\mathscr{S}(x)\\d\mid a}}1
\ge1.
\end{equation}
On the other hand, if \cref{Srichan:wrong_bound} holds, since $1\le d\le x$ now, we have
\[
1
\le
\sum_{\substack{a\in\mathscr{S}(x)\\d\mid a}}1
\ll
\frac{x^{\frac{1}{2}}}{d}
=
\frac{N^{\frac{3}{2}}}{N^{2}}
=
N^{-\frac{1}{2}},
\]
which is a contradiction for large $N$.
This shows the bound \cref{Srichan:wrong_bound} does not hold for large $d$.
\end{example}
Although we cannot apply \cref{prethm:WuYu} literally,
the result of Wu and Yu holds even for $d>x^{\frac{1}{4}}(\log x)^{-\frac{3}{2}}$
as we justify in the proof of \cref{thm:WuYu_improved}.
Thus, if we ignore the $\log x$ factor (which can be dealt with after replaced by $\log\frac{x}{d}$),
\cref{prethm:WuYu} gives the bound
\[
\sum_{\substack{a\in\mathscr{S}(x)\\d\mid a}}1
\ll
\frac{\sqrt{x}}{d}
+
\biggl(\frac{x}{d}\biggr)^{\frac{1}{3}}.
\]
Thus, we may bound \cref{Srichan:failed_sum} as
\[
\ll
x\sum_{x^{\frac{1}{4}}(\log x)^{-\frac{3}{2}}<d\le x}\frac{1}{d^{2}}
+
x^{\frac{2}{3}}
\sum_{x^{\frac{1}{4}}(\log x)^{-\frac{3}{2}}<d\le x}\frac{1}{d^{\frac{2}{3}}}
\asymp
x,
\]
which is slightly short of giving any asymptotic formula for the left-hand side of \cref{Srichan:failed_result}.
Therefore, it seems that Srichan's original proof does not work as it is.
The problem here is somehow related to the sparseness of $\mathscr{S}(x)$.

We succeed in recovering and improving Srichan's result
by introducing an averaging trick and our improved version of the result of Wu and Yu~\cref{thm:WuYu_improved}:
\begin{theorem}
\label{thm:primitive_Sx}
For an exponent pairt $(k,\ell)$ and $x\ge2$, we have
\[
\sum_{\substack{m,n\in\mathscr{S}(x)\\(m,n)=1}}1
=
\frac{4}{\zeta(2)}x
+
O(
x^{\frac{1}{2}+\frac{k+\ell}{2k+2}}
),
\]
where the implicit constant depends on $(k,\ell)$.
\end{theorem}
Note that $(k,\ell)=(\frac{1}{2},\frac{1}{2})$ recovers the original claim of Srichan.

We also prove a result with counting the multiplicity:
\begin{theorem}
\label{thm:primitive_xn}
For an exponent pairt $(k,\ell)$ and $x\ge2$, we have
\[
\sum_{\substack{
1\le m,n\le x\\
([x/m],[x/n])=1
}}1
=
\mathfrak{S}x^{2}
+
O(
x^{1+\frac{k+\ell}{2k+2}}
),
\]
where
\[
\mathfrak{S}
\coloneqq
\sum_{\substack{m,n=1\\(m,n)=1}}^{\infty}
\frac{1}{mn(m+1)(n+1)}
\]
and the implicit constant depends on $(k,\ell)$.
\end{theorem}

We shall prove \cref{thm:primitive_Sx} and \cref{thm:primitive_xn}
in \cref{sec:primitive_lattice_points}.

\subsection{``Multiplicative'' Titchmarsh divisor problem}
\label{subsec:multiplicative_Titchmarsh_divisor_problem}
A variant of the well-known Titchmarsh divisor problem states
\begin{equation}
\label{Titchmarsh_divisor_problem}
\sum_{p<N}\tau(N-p)
=
\prod_{p\mid N}\biggl(1+\frac{1}{p^{2}-p+1}\biggr)
\frac{\zeta(2)\zeta(3)}{\zeta(6)}N
+
O\biggl(\frac{N\log\log N}{\log N}\biggr)
\end{equation}
for positive integers $N\ge4$, where $\tau(n)$ is the usual divisor function,
i.e.\ the number of divisors of $n$.
For the original formulation see~\cite{Titchmarsh:Divisor}
and for the unconditional proof,
see \cite[Chapter~8]{Linnik:dispersion}
or \cite{Halberstam:TitchmarshDivisor}.
The sum in \cref{Titchmarsh_divisor_problem} can be seen as
``the additive convolution of $\tau$ and the indicator function of primes''.
If we consider its multiplicative analog
without concerning the divisibility, we may arrive at the sum
\[
\sum_{p\le x}\tau\biggl(\biggl[\frac{x}{p}\biggr]\biggr)
\]
or, more generally, the sum
\[
\sum_{p\le x}a\biggl(\biggl[\frac{x}{p}\biggr]\biggr)
\]
for a given arithmetic function $a(n)$.
Indeed, because of its unconventional nature,
the difference of the indicator function of primes
and the von Mangoldt function is not negligible anymore
for some arithmetic function $a(n)$.
Hence, we consider the following two sums separately:
\begin{equation}
\label{multiplicative_titchmarsh_divisor_problem}
\sum_{n\le x}\Lambda(n)a\biggl(\biggl[\frac{x}{n}\biggr]\biggr)
\and
\sum_{p\le x}a\biggl(\biggl[\frac{x}{p}\biggr]\biggr).
\end{equation}
We shall call the problems related to such sums
``the multiplicative Titchmarsh divisor problems''.

We first consider the easy case where $a(n)$ is small:
\begin{theorem}
\label{thm:MTD_easy_with_vonMangoldt}
Let $a\colon\mathbb{N}\to\mathbb{C}$ be an arithmetic function satisfying
\begin{enumerate}[label=\textup{(\alph*)}]
\item\label{thm:MTD_easy_with_vonMangoldt:a_mean}
For $x\ge 2$, we have
\[
\sum_{n\le x}\frac{|a(n)|}{n}
\ll
(\log x)^{A}
\]
with some constant $A\ge0$.
\item\label{thm:MTD_easy_with_vonMangoldt:sup}
For $n\in\mathbb{N}$, we have
\[
|a(n)|\le Bn^{\theta}
\]
with some constants $B\ge0$ and $\theta\in[0,\frac{1}{2})$.
\end{enumerate}
Then, for $x\ge1$, we have
\[
\sum_{n\le x}\Lambda(n)a\biggl(\biggl[\frac{x}{n}\biggr]\biggr)
=
x\sum_{n=1}^{\infty}\frac{a(n)}{n(n+1)}
+O(x\exp(-c\sqrt{\log x})),
\]
where $c>0$ is some absolute constant and the implicit consant depends only on $A,B,\theta$.
\end{theorem}

\begin{theorem}
\label{thm:MTD_easy_unweighted}
Under the same condition as in \cref{thm:MTD_easy_with_vonMangoldt}, for $x\ge2$ and $K\in\mathbb{Z}_{\ge0}$, we have
\begin{align}
\sum_{p\le x}a\biggl(\biggl[\frac{x}{p}\biggr]\biggr)
=
\frac{x}{\log x}
\sum_{k=0}^{K}\biggl(\frac{1}{\log x}\biggr)^{k}
\int_{1}^{\infty}\frac{a([u])(\log u)^{k}}{u^2}du
+
O\biggl(\frac{x}{\log x}\biggl(\frac{1}{\log x}\biggr)^{K+1}\biggr),
\end{align}
where the implicit constant depends only on $A,B,\theta,K$.
\end{theorem}

We remark that the condition \cref{thm:MTD_easy_with_vonMangoldt:a_mean}
may be weakened by improving or using the error estimate of the prime number theorem more carefully.

We next consider the more difficult case $a(n)\asymp n$.
With $a(n)=\phi(n)$ and the von Mangoldt function,
this is now straightforward after the success of Zhai~\cite{Zhai,Zhai:corrigendum}:
\begin{theorem}
\label{thm:phi_xn_Titchmarsh_vonMangoldt}
For $x\ge 27$, we have
\[
\sum_{n\le x}\Lambda(n)\phi\biggl(\biggl[\frac{x}{n}\biggr]\biggr)
=
\frac{1}{\zeta(2)}x\log x
+
O(x(\log x)^{\frac{2}{3}}(\log\log x)^{\frac{1}{3}}).
\]
\end{theorem}

\begin{remark}
\label{rem:no_large_sieve}
We somehow simplify the method of Zhai~\cite{Zhai,Zhai:corrigendum}
by using the estimate of exponential sum with arithmetic functions \cref{lem:expsum_KVW}.
This allows one not to use the double large sieve~\cite[Lemma~2.4]{BombieriIwaniec:ZetaOrder}
as in \cite{Zhai:corrigendum} and to use a rather simple decomposition.
\end{remark}

The ``new'' ingredient in the proof of \cref{thm:phi_xn_Titchmarsh_vonMangoldt}
is just the estimate of the exponential sum
\begin{equation}
\label{KVW_exponential_sum}
\sum_{P<n\le 2P}\Lambda(n)e\biggl(\frac{Q}{n}\biggr)
\end{equation}
\`{a} la Korobov--Vinogradov--Walfisz,
which is essentially due to H. Q. Liu~\cite{HQLiu}
and the second author gave a detailed argument for the above particular sum \cref{KVW_exponential_sum} in \cite{Yuta:Walfisz}.
It thus seems, at first sight, the same for the case
\begin{equation}
\label{MTD_phi_uneighted}
\sum_{p\le x}\phi\biggl(\biggl[\frac{x}{p}\biggr]\biggr).
\end{equation}
However, there is a notable difference
\begin{align}
\sum_{n\le P}\Lambda(n)\phi\biggl(\biggl[\frac{x}{n}\biggr]\biggr)
&\ll
x\sum_{n\le P}\frac{\Lambda(n)}{n}
\ll
x\log P
\quad\text{vs.}\quad
\sum_{p\le P}\phi\biggl(\biggl[\frac{x}{p}\biggr]\biggr)
\ll
x\sum_{p\le x}\frac{1}{p}
\ll
x\log\log P.
\end{align}
With $P=\exp(O((\log x)^{\frac{2}{3}}(\log\log x)^{\frac{1}{3}}))$,
we can discard the first sum not affecting the main term of the size $\asymp x\log x$.
However, in the latter case, with $P=\exp(O((\log x)^{\frac{2}{3}}(\log\log x)^{\frac{1}{3}}))$,
the second sum is of the same order as the main term of the size $\asymp x\log\log x$.
Note that the level $\exp(O((\log x)^{\frac{2}{3}}(\log\log x)^{\frac{1}{3}}))$
is required to estimate the exponential sum \cref{KVW_exponential_sum}.
This thus gives a new difficulty and we could only give upper and lower bounds
for the sum \cref{MTD_phi_uneighted}:
\begin{theorem}
\label{thm:phi_xn_Titchmarsh}
For large $x$, we have
\begin{align}
\sum_{p\le x}\phi\biggl(\biggl[\frac{x}{p}\biggr]\biggr)
&\le
\biggl(\frac{1}{\zeta(2)}+\frac{2}{3}\biggl(1-\frac{1}{\zeta(2)}\biggr)
+O\biggl(\frac{\log\log\log x}{\log\log x}\biggr)\biggr)
x\log\log x,\\
\sum_{p\le x}\phi\biggl(\biggl[\frac{x}{p}\biggr]\biggr)
&\ge
\biggl(\frac{1}{\zeta(2)}-\frac{2}{3}\frac{1}{\zeta(2)}
+O\biggl(\frac{\log\log\log x}{\log\log x}\biggr)\biggr)x\log\log x.
\end{align}
\end{theorem}

We thus conjecture the following asymptotic formula
though we think currently that this problem is out of reach
according to the above observation and the current limited technique in exponential
sum of the type \cref{KVW_exponential_sum}.
\begin{conjecture}
\label{conj:MTD_phi_unweighted}
As $x\to\infty$, we have
\[
\sum_{p\le x}\phi\biggl(\biggl[\frac{x}{p}\biggr]\biggr)
\sim
\frac{1}{\zeta(2)}x\log\log x.
\]
\end{conjecture}

On the contrary,
by the conjecture of Montgomery~\cite[Formula~(9)]{Montgomery:FluctuationEulerPhi}:
\[
\sum_{n\le x}\phi(n)
=
\frac{1}{2\zeta(2)}x^{2}
+
\Omega_{\pm}(x\log\log x)
\quad(x\to\infty),
\]
\cref{conj:MTD_phi_unweighted} can be rather doubtful.

We shall prove
\cref{thm:MTD_easy_with_vonMangoldt,thm:MTD_easy_unweighted}
in \cref{sec:xp_small}
and \cref{thm:phi_xn_Titchmarsh_vonMangoldt,thm:phi_xn_Titchmarsh}
in \cref{sec:xp_phi}.

\section{Notation}
\label{sec:notation}
Throughout the paper
$x,y,z,u,A,B,H,Q,U,X,s,\epsilon$ denote positive real numbers,
$\theta,\delta,\eta$ denote real numbers,
$m,n,d,q,P,\nu$ denote positive integers and
$a,b,h,K$ denote integers.
The letter $N$ denote a positive integer or a positive real number.
The letter $k$ denote a positive integer
or the member of an exponent pair $(k,\ell)$.
The letter $p$ is reserved for prime numbers
or the degrees of differentiation.
The letter $c$ is used for positive constants
which can take different value line by line.
The functions $\tau(n),\mu(n),\phi(n),\Lambda(n)$ stand for
the divisor function (the number of divisors of $n$),
the M\"obius function,
the Euler totient function and
the von Mangoldt function.
The function $\zeta(s)$ stands for the Riemann zeta function.

For integers $m$ and $n$, we write $(m,n)$ for their greatest common divisor,
which can be easily distinguished from the pair $(m,n)$ by the context.

For a real number $x$, the symbol $[x]$ stands for the largest integer $\le x$
and $\{x\}\coloneqq x-[x]$ is the fractional part of $x$.
The function $\psi(x)$ is defined by
\[
\psi(x)\coloneqq\{x\}-\tfrac{1}{2}.
\]
which is the sawtooth function. Also, we write $e(x)\coloneqq\exp(2\pi ix)$.

For $x\in\mathbb{R}_{\ge1}$, we let
\[
\mathscr{S}(x)
\coloneqq
\{[x/n]\mid n\in[1,x]\cap\mathbb{Z}\}.
\]
For $x\in\mathbb{R}_{\ge1}$, $a\in\mathbb{Z}$ and $q\in\mathbb{N}$, we let
\[
S(x,q,a)
\coloneqq
\sum_{\substack{
n\le x\\\relax
[x/n]\equiv a\ \mod{q}
}}1
\and
S^{\ast}(x,q,a)
\coloneqq
\sum_{\substack{
n\in\mathscr{S}(x)\\
n\equiv a\ \mod{q}
}}
1.
\]

For a logical formula $P$,
we write $\mathbbm{1}_{P}$
for the indicator function of $P$.

If Theorem or Lemma is stated
with the phrase ``where the implicit constants depend on $a,b,c,\ldots$'',
then every implicit constant in the corresponding proof
may also depend on $a,b,c,\ldots$ even without special mentions.

\section{Distribution of the \texorpdfstring{$[x/n]$}{[x/n]}-sequence in arithmetic progressions}
\label{sec:distribution_AP}
In this section, we prove \cref{thm:WuYu_improved}.
We first expand the error terms of $S(x,q,a)$ and $S^{\ast}(x,q,a)$
into the sum over the sawtooth function $\psi$.
Fortunately, we get the same type of the sum with $\psi$
for both of $S(x,q,a)$ and $S^{\ast}(x,q,a)$.

\begin{proposition}
\label{prop:S_d_decomp}
For $x\ge1$, $a,q\in\mathbb{Z}$ with $1\le a\le q$ and $q\ge1$, we have
\[
S(x,q,a)
=
x\sum_{\substack{n=1\\n\equiv a\ \mod{q}}}^{\infty}\frac{1}{n(n+1)}
+R_{1}+R_{2}+O(1),
\]
where
\begin{align}
R_{1}
&\coloneqq
-
\sum_{n\le\sqrt{\frac{x}{q}}}
\biggl(\psi\biggl(\frac{x}{qn+a}\biggr)-\psi\biggl(\frac{x}{qn+a+1}\biggr)\biggr),\\
R_{2}
&\coloneqq
-
\sum_{n\le\sqrt{\frac{x}{q}}}
\biggl(\psi\biggl(\frac{x}{qn}-\frac{a}{q}\biggr)-\psi\biggl(\frac{x}{qn}-\frac{a+1}{q}\biggr)\biggr).
\end{align}
\end{proposition}
\begin{proof}
We first show the case $q\ge x$.
In this case, $R_{1},R_{2}\ll1$.
If $a>x$, then the left-hand side of the assertion is $0$
and the first term on the right-hand side is
\[
x\sum_{\substack{n=1\\n\equiv a\ \mod{q}}}^{\infty}\frac{1}{n(n+1)}
\ll
x\sum_{\substack{n\ge a}}\frac{1}{n(n+1)}
\ll
\frac{x}{a}
\ll
1
\]
since $x<a\le q$. Thus, the case $q\ge x$ and $a>x$ is trivial.
If $1\le a\le x$, the left-hand side is reduced to
\[
S(x,q,a)
=
\sum_{\substack{n\le x\\\relax[x/n]=a}}1
=
\sum_{\substack{n\in\mathbb{N}\\\relax[x/n]=a}}1
\]
since $q\ge x$. By using
\[
\biggl[\frac{x}{n}\biggr]=a
\iff
a\le\frac{x}{n}<a+1
\iff
\frac{x}{a+1}<n\le\frac{x}{a},
\]
we then have
\[
S(x,q,a)
=
\biggl[\frac{x}{a}\biggr]-\biggl[\frac{x}{a+1}\biggr]
=
\frac{x}{a(a+1)}+O(1).
\]
On the other hand, the first term on the right-hand side is
\[
x\sum_{\substack{n=1\\n\equiv a\ \mod{q}}}^{\infty}\frac{1}{n(n+1)}
=
\frac{x}{a(a+1)}
+
x\sum_{\substack{n>q\\n\equiv a\ \mod{q}}}\frac{1}{n(n+1)}
=
\frac{x}{a(a+1)}+O(1).
\]
Thus, the assertion holds if $q\ge x$ and $1\le a\le x$.

We then prove the case $1\le q<x$. We have
\begin{align}
S(x,q,a)
&=
\sum_{\substack{\nu\le x\\\nu\equiv a\ \mod{q}}}
\sum_{\substack{n\le x\\\relax[x/n]=\nu}}
1
=
\sum_{\substack{\nu\le x\\\nu\equiv a\ \mod{q}}}
\sum_{\substack{n\le x\\\relax\nu\le x/n<\nu+1}}
1\\
&=
\sum_{\substack{\nu\le\sqrt{xq}\\\nu\equiv a\ \mod{q}}}
\sum_{\substack{n\le x\\\relax\nu\le x/n<\nu+1}}
1
+
\sum_{\substack{\sqrt{xq}<\nu\le x\\\nu\equiv a\ \mod{q}}}
\sum_{\substack{n\le x\\\relax\nu\le x/n<\nu+1}}
1
\eqqcolon
S^{(1)}
+
S^{(2)},\quad\text{say},
\end{align}
the decomposition of which is valid since $1\le q<x$. For $S^{(1)}$, we have
\begin{align}
S^{(1)}
&=
\sum_{\substack{\nu\le\sqrt{xq}\\\nu\equiv a\ \mod{q}}}
\sum_{\frac{x}{\nu+1}<n\le\min(\frac{x}{\nu},x)}
1\\
&=
\sum_{\substack{\nu\le\sqrt{xq}\\\nu\equiv a\ \mod{q}}}
\biggl(\biggl[\frac{x}{\nu}\biggr]-\biggl[\frac{x}{\nu+1}\biggr]\biggr)\\
&=
\sum_{\substack{n\le\sqrt{xq}\\n\equiv a\ \mod{q}}}\frac{x}{n(n+1)}
-
\sum_{\substack{n\le\sqrt{xq}\\n\equiv a\ \mod{q}}}
\biggl(\psi\biggl(\frac{x}{n}\biggr)-\psi\biggl(\frac{x}{n+1}\biggr)\biggr)\\
&=
\sum_{\substack{n\le\sqrt{xq}\\n\equiv a\ \mod{q}}}\frac{x}{n(n+1)}
-
\sum_{0\le n\le\sqrt{\frac{x}{q}}-\frac{a}{q}}
\biggl(\psi\biggl(\frac{x}{qn+a}\biggr)-\psi\biggl(\frac{x}{qn+a+1}\biggr)\biggr)\\
&=
\sum_{\substack{n\le\sqrt{xq}\\n\equiv a\ \mod{q}}}\frac{x}{n(n+1)}
-
\sum_{n\le\sqrt{\frac{x}{q}}}
\biggl(\psi\biggl(\frac{x}{qn+a}\biggr)-\psi\biggl(\frac{x}{qn+a+1}\biggr)\biggr)
+
O(1),
\end{align}
where we used $1\le a\le q$ in the last equality.
For $S^{(2)}$, we have
\begin{align}
S^{(2)}
&=
\sum_{\substack{\sqrt{xq}<\nu\le x\\\nu\equiv a\ \mod{q}}}
\sum_{\substack{n\le x\\\relax\nu\le x/n<\nu+1}}
1
=
\sum_{n\le\sqrt{\frac{x}{q}}}
\sum_{\substack{\frac{x}{n}-1<\nu\le\frac{x}{n}\\\sqrt{xq}<\nu\le x\\\nu\equiv a\ \mod{q}}}
1.
\end{align}
Note that
\[
n\le\sqrt{\frac{x}{q}}-1
\implies
\frac{x}{n}-1
\ge
\sqrt{xq}
\]
since $1\le q<x$ implies
\[
\frac{x}{\sqrt{\frac{x}{q}}-1}-1
\ge
\sqrt{xq}
\iff
x
\ge
(\sqrt{xq}+1)(\sqrt{x/q}-1)
\iff
x\ge
x-\sqrt{xq}+\sqrt{x/q}-1.
\]
Therefore, since $q\ge1$, we have
\begin{align}
S^{(2)}
=
\sum_{n\le\sqrt{\frac{x}{q}}}
\sum_{\substack{\frac{x}{n}-1<\nu\le\frac{x}{n}\\\nu\equiv a\ \mod{q}}}
1
+
O(1).
\end{align}
We now have
\begin{align}
\sum_{n\le\sqrt{\frac{x}{q}}}
\sum_{\substack{\frac{x}{n}-1<\nu\le\frac{x}{n}\\\nu\equiv a\ \mod{q}}}
1
&=
\sum_{n\le\sqrt{\frac{x}{q}}}
\sum_{\frac{x}{qn}-\frac{a+1}{q}<\nu\le\frac{x}{qn}-\frac{a}{q}}
1\\
&=
\frac{\sqrt{x}}{q^{\frac{3}{2}}}
-
\sum_{n\le\sqrt{\frac{x}{q}}}
\biggl(\psi\biggl(\frac{x}{qn}-\frac{a}{q}\biggr)-\psi\biggl(\frac{x}{qn}-\frac{a+1}{q}\biggr)\biggr).
\end{align}
Therefore, we have
\[
S^{(2)}
=
\frac{\sqrt{x}}{q^{\frac{3}{2}}}
+R_{2}
+O(1).
\]
By combining the above results, we get
\[
S(x,q,a)
=
\sum_{\substack{n\le\sqrt{xq}\\n\equiv a\ \mod{q}}}\frac{x}{n(n+1)}
+
\frac{\sqrt{x}}{q^{\frac{3}{2}}}
+
R_{1}+R_{2}+O(1).
\]
We now have
\begin{align}
\sum_{\substack{n>\sqrt{xq}\\n\equiv a\ \mod{q}}}\frac{x}{n(n+1)}
&=
\sum_{n>\sqrt{\frac{x}{q}}-\frac{a}{q}}\frac{x}{(qn+a)(qn+a+1)}\\
&=
\sum_{n>\sqrt{\frac{x}{q}}-\frac{a}{q}}\frac{x}{(qn+a)^{2}}
+
O\biggl(\sum_{n>\sqrt{\frac{x}{q}}-\frac{a}{q}}\frac{x}{(qn+a)^{3}}\biggr)\\
&=
\sum_{n>\sqrt{\frac{x}{q}}-\frac{a}{q}}\frac{x}{(qn+a)^{2}}
+
O\biggl(\frac{x}{q^{3}}\sum_{n>\sqrt{\frac{x}{q}}-\frac{a}{q}}\frac{1}{(n+\frac{a}{q})^{3}}\biggr)\\
&=
\frac{x}{q^{2}}\sum_{n>\sqrt{\frac{x}{q}}-\frac{a}{q}}\frac{1}{(n+\frac{a}{q})^{2}}
+
O\biggl(\frac{1}{q^{2}}\biggr)\\
&=
\frac{x}{q^{2}}\int_{\sqrt{\frac{x}{q}}}^{\infty}\frac{du}{u^{2}}
+
O\biggl(\frac{1}{q}\biggr)\\
&=
\frac{\sqrt{x}}{q^{\frac{3}{2}}}
+
O\biggl(\frac{1}{q}\biggr)
\end{align}
and so the assertion follows.
\end{proof}

\begin{proposition}
\label{prop:S_d_ast_decomp_pre}
For $x\ge1$, $a,q\in\mathbb{Z}$ with $1\le a\le q$ and $q\ge1$, we have
\[
S^{\ast}(x,q,a)
=
\frac{\sqrt{x}}{q}
+
\sum_{\frac{\sqrt{x}}{q}-\frac{a}{q}<n\le\frac{x}{q}-\frac{a}{q}}
\biggl(
\biggl[\frac{x}{qn+a}\biggr]
-
\biggl[\frac{x}{qn+a+1}\biggr]
\biggr)
+O(1).
\]
\end{proposition}
\begin{proof}
For $n\in[1,x]\cap\mathbb{N}$, we have
\begin{align}
n\in\mathscr{S}(x)
&\iff
\exists \nu\in[1,x]\cap \mathbb{Z}\ \text{s.t.}\ n=\biggl[\frac{x}{\nu}\biggr]\\
&\iff
\exists \nu\in[1,x]\cap \mathbb{Z}\ \text{s.t.}\ n\le\frac{x}{\nu}<n+1\\
&\iff
\exists \nu\in[1,x]\cap \mathbb{Z}\ \text{s.t.}\ \frac{x}{n+1}<\nu\le\frac{x}{n}\\
&\iff
\exists \nu\in\mathbb{Z}\ \text{s.t.}\ \frac{x}{n+1}<\nu\le\frac{x}{n}\\
&\iff
\biggl[\frac{x}{n}\biggr]
-
\biggl[\frac{x}{n+1}\biggr]>0.
\end{align}
Therefore, we have
\begin{align}
S^{\ast}(x,q,a)
=
\sum_{\substack{n\le x\\n\equiv a\ \mod{q}}}
\mathbbm{1}_{[\frac{x}{n}]-[\frac{x}{n+1}]>0}
&=
\sum_{\substack{n\le\sqrt{x}-1\\n\equiv a\ \mod{q}}}
+
\sum_{\substack{\sqrt{x}<n\le x\\n\equiv a\ \mod{q}}}
{}+{}
O(1)\\
&=
S^{(1)}
+
S^{(2)}
+
O(1),\quad\text{say}.
\end{align}
In $S^{(1)}$, we have $n\le\sqrt{x}-1$ and so
\[
\biggl[\frac{x}{n}\biggr]-\biggl[\frac{x}{n+1}\biggr]
>
\frac{x}{n}-\frac{x}{n+1}-1
=
\frac{x}{n(n+1)}-1
>
\frac{x}{(n+1)^{2}}-1
\ge
0.
\]
This gives
\[
S^{(1)}
=
\sum_{\substack{n\le\sqrt{x}-1\\n\equiv a\ \mod{q}}}
\mathbbm{1}_{[\frac{x}{n}]-[\frac{x}{n+1}]>0}
=
\sum_{\substack{n\le\sqrt{x}-1\\n\equiv a\ \mod{q}}}
1
=
\frac{\sqrt{x}}{q}+O(1).
\]
On the other hand, in $S^{(2)}$, we have $n>\sqrt{x}$ and so
\[
\biggl[\frac{x}{n}\biggr]-\biggl[\frac{x}{n+1}\biggr]
<
\frac{x}{n}-\frac{x}{n+1}+1
=
\frac{x}{n(n+1)}+1
<
\frac{x}{n^{2}}+1
<
2.
\]
This gives
\[
\biggl[\frac{x}{n}\biggr]-\biggl[\frac{x}{n+1}\biggr]
\in\{0,1\}
\quad\text{and so}\quad
\mathbbm{1}_{[\frac{x}{n}]-[\frac{x}{n+1}]>0}
=
\biggl[\frac{x}{n}\biggr]-\biggl[\frac{x}{n+1}\biggr]
\]
in $S^{(2)}$. Therefore, we have
\[
S^{(2)}
=
\sum_{\substack{\sqrt{x}<n\le x\\n\equiv a\ \mod{q}}}
\biggl(\biggl[\frac{x}{n}\biggr]-\biggl[\frac{x}{n+1}\biggr]\biggr)
=
\sum_{\substack{\frac{\sqrt{x}}{q}-\frac{a}{q}<n\le\frac{x}{q}-\frac{a}{q}}}
\biggl(\biggl[\frac{x}{qn+a}\biggr]-\biggl[\frac{x}{qn+a+1}\biggr]\biggr).
\]
This completes the proof.
\end{proof}

\begin{proposition}
\label{prop:S_d_ast_decomp}
For $x\ge1$, $a,q\in\mathbb{Z}$ with $1\le a\le q$ and $q\ge1$, we have
\[
S^{\ast}(x,q,a)
=
\frac{2\sqrt{x}}{q}
+
R_{1}^{\ast}
+
R_{2}
+O(1),
\]
where
\begin{align}
R_{1}^{\ast}
&\coloneqq
-
\sum_{\frac{\sqrt{x}}{q}<n\le\sqrt{\frac{x}{q}}}
\biggl(\psi\biggl(\frac{x}{qn+a}\biggr)-\psi\biggl(\frac{x}{qn+a+1}\biggr)\biggr),\\
R_{2}
&\coloneqq
-
\sum_{n\le\sqrt{\frac{x}{q}}}
\biggl(
\psi\biggl(\frac{x}{qn}-\frac{a}{q}\biggr)
-
\psi\biggl(\frac{x}{qn}-\frac{a+1}{q}\biggr)
\biggr).
\end{align}
\textup{(}Note that $R_{2}$ is the same as in \cref{prop:S_d_decomp}.\textup{)}
\end{proposition}
\begin{proof}
We first consider the case $q\ge x$.
In this case,
the left-hand side is $\ll1$ since there are at most one integer in $[1,x]$
congruent to $a\ \mod{q}$. Also,
all terms on the right-hand side is $\ll1$.
Thus, the assertion is trivial for the case $q\ge x$.

We next consider the case $1\le q<x$.
By \cref{prop:S_d_ast_decomp_pre}, it suffices to show
\[
T
\coloneqq
\sum_{\frac{\sqrt{x}}{q}-\frac{a}{q}<n\le\frac{x}{q}-\frac{a}{q}}
\biggl(
\biggl[\frac{x}{qn+a}\biggr]
-
\biggl[\frac{x}{qn+a+1}\biggr]
\biggr)
=
\frac{\sqrt{x}}{q}
+
R_{1}^{\ast}
+
R_{2}
+
O(1).
\]
We first decompose the sum as
\[
T
=
\sum_{\frac{\sqrt{x}}{q}-\frac{a}{q}<n\le\sqrt{\frac{x}{q}}-\frac{a}{q}}
+
\sum_{\sqrt{\frac{x}{q}}-\frac{a}{q}<n\le\frac{x}{q}-\frac{a}{q}}
=
T^{(1)}+T^{(2)},
\quad\text{say},
\]
the decomposition of which is valid since $1\le q<x$.
For the sum $T^{(1)}$, we have
\begin{align}
T^{(1)}
&=
\sum_{\frac{\sqrt{x}}{q}-\frac{a}{q}<n\le\sqrt{\frac{x}{q}}-\frac{a}{q}}
\frac{x}{(qn+a)(qn+a+1)}\\
&\hspace{30mm}-
\sum_{\frac{\sqrt{x}}{q}-\frac{a}{q}<n\le\sqrt{\frac{x}{q}}-\frac{a}{q}}
\biggl(
\psi\biggl(\frac{x}{qn+a}\biggr)
-
\psi\biggl(\frac{x}{qn+a+1}\biggr)
\biggr)\\
&=
\sum_{\frac{\sqrt{x}}{q}-\frac{a}{q}<n\le\sqrt{\frac{x}{q}}-\frac{a}{q}}
\frac{x}{(qn+a)(qn+a+1)}\\
&\hspace{30mm}
-
\sum_{\frac{\sqrt{x}}{q}<n\le\sqrt{\frac{x}{q}}}
\biggl(
\psi\biggl(\frac{x}{qn+a}\biggr)
-
\psi\biggl(\frac{x}{qn+a+1}\biggr)
\biggr)
+O(1),
\end{align}
where we used $1\le a\le q$ in the last equality.
For this main term, we have
\begin{align}
&\sum_{\frac{\sqrt{x}}{q}-\frac{a}{q}<n\le\sqrt{\frac{x}{q}}-\frac{a}{q}}
\frac{x}{(qn+a)(qn+a+1)}\\
&=
\sum_{\frac{\sqrt{x}}{q}-\frac{a}{q}<n\le\sqrt{\frac{x}{q}}-\frac{a}{q}}
\frac{x}{(qn+a)^{2}}
+
O\biggl(
\sum_{\frac{\sqrt{x}}{q}-\frac{a}{q}<n\le\sqrt{\frac{x}{q}}-\frac{a}{q}}
\frac{x}{(qn+a)^{3}}
\biggr)\\
&=
\sum_{\frac{\sqrt{x}}{q}-\frac{a}{q}<n\le\sqrt{\frac{x}{q}}-\frac{a}{q}}
\frac{x}{(qn+a)^{2}}
+
O\biggl(
\frac{x}{q^{3}}
\sum_{\frac{\sqrt{x}}{q}-\frac{a}{q}<n\le\sqrt{\frac{x}{q}}-\frac{a}{q}}
\frac{1}{(n+\frac{a}{q})^{3}}
\biggr)\\
&=
\sum_{\frac{\sqrt{x}}{q}-\frac{a}{q}<n\le\sqrt{\frac{x}{q}}-\frac{a}{q}}
\frac{x}{(qn+a)^{2}}
+
O\biggl(
\frac{1}{q}
\biggr)\\
&=
\frac{x}{q^{2}}
\int_{\frac{\sqrt{x}}{q}}^{\sqrt{\frac{x}{q}}}
\frac{du}{u^{2}}
+
O(1).
\end{align}
We thus have
\begin{equation}
T^{(1)}
=
\frac{x}{q^{2}}
\int_{\frac{\sqrt{x}}{q}}^{\sqrt{\frac{x}{q}}}
\frac{du}{u^{2}}
+
R_{1}^{\ast}
+
O(1).
\end{equation}
For the sum $T^{(2)}$, we have
\begin{align}
T^{(2)}
=
\sum_{\sqrt{\frac{x}{q}}-\frac{a}{q}<\nu\le\frac{x}{q}-\frac{a}{q}}
\biggl(
\biggl[\frac{x}{q\nu+a}\biggr]
-
\biggl[\frac{x}{q\nu+a+1}\biggr]
\biggr)
&=
\sum_{\sqrt{\frac{x}{q}}-\frac{a}{q}<\nu\le\frac{x}{q}-\frac{a}{q}}
\sum_{\frac{x}{q\nu+a+1}<n\le\frac{x}{q\nu+a}}
1\\
&=
\sum_{n\le\sqrt{\frac{x}{q}}}
\sum_{\substack{
\sqrt{\frac{x}{q}}-\frac{a}{q}<\nu\le\frac{x}{q}-\frac{a}{q}\\
\frac{x}{qn}-\frac{a+1}{q}<\nu\le\frac{x}{qn}-\frac{a}{q}
}}
1.
\end{align}
Note that
\[
n\le\sqrt{\frac{x}{q}}-1
\implies
\frac{x}{qn}-\frac{a+1}{q}\ge\sqrt{\frac{x}{q}}-\frac{a}{q}
\]
since $1\le q<x$ implies
\begin{align}
\frac{x}{q(\sqrt{\frac{x}{q}}-1)}-\frac{a+1}{q}\ge\sqrt{\frac{x}{q}}-\frac{a}{q}
&\iff
x
\ge
(\sqrt{xq}+1)
(\sqrt{x/q}-1)\\
&\iff
x
\ge
x-\sqrt{xq}+\sqrt{x/q}-1.
\end{align}
Therefore, we have
\begin{align}
T^{(2)}
&=
\sum_{n\le\sqrt{\frac{x}{q}}}
\sum_{\frac{x}{qn}-\frac{a+1}{q}<\nu\le\frac{x}{qn}-\frac{a}{q}}
1
+O(1)\\
&=
\frac{\sqrt{x}}{q^{\frac{3}{2}}}
-
\sum_{n\le\sqrt{\frac{x}{q}}}
\biggl(
\psi\biggl(\frac{x}{qn}-\frac{a}{q}\biggr)
-
\psi\biggl(\frac{x}{qn}-\frac{a+1}{q}\biggr)
\biggr)
+O(1).
\end{align}
By combining the above results, we have
\begin{align}
T
=
\frac{x}{q^{2}}
\int_{\frac{\sqrt{x}}{q}}^{\sqrt{\frac{x}{q}}}
\frac{du}{u^{2}}
+
\frac{\sqrt{x}}{q^{\frac{3}{2}}}
+
R_{1}^{\ast}
+
R_{2}
+
O(1)
=
\frac{\sqrt{x}}{q}
+
R_{1}^{\ast}
+
R_{2}
+
O(1).
\end{align}
This completes the proof.
\end{proof}


We now prepare some tools on the exponential sums.

\begin{lemma}[Vaaler's approximation]
\label{lem:Vaaler}
For $H>0$,
there is a Fourier polynomial
\[
\psi_{H}(x)=-\sum_{1\le|h|<H}\frac{c_{H}(h)}{2\pi ih}e(hx)
\]
with real numbers $(c_{H}(h))_{1\le|h|<H}$ such that
\begin{equation}
\label{lem:Vaaler:error}
|\psi(x)-\psi_H(x)|\le\frac{1}{2H}\sum_{|h|<H}\left(1-\frac{|h|}{H}\right)e(hx)
\quad\text{for}\quad x\in\mathbb{R}
\end{equation}
and
\begin{equation}
\label{lem:Vaaler:c_bound}
|c_{H}(h)|\le1\quad\text{for $h\neq0$}.
\end{equation}
\end{lemma}
\begin{proof}
This follows by the same construction as Theorem~A.6 of \cite{GrahamKolesnik}.
\end{proof}

\begin{lemma}
\label{lem:Erdos_Turan}
For a finite sequence $(x_{n})_{n=1}^{N}$ of real numbers and $H>0$, we have
\[
\biggl|\sum_{n=1}^{N}\psi(x_{n})\biggr|
\ll
\sum_{1\le h\le H}\frac{1}{h}
\biggl|
\sum_{n=1}^{N}e(hx_{n})
\biggr|
+
\frac{N}{H},
\]
where the implicit constant is absolute.
\end{lemma}
\begin{proof}
We use $\psi_{H}$ given in \cref{lem:Vaaler}.
We have
\begin{align}
\biggl|\sum_{n=1}^{N}(\psi(x_{n})-\psi_{H}(x_{n}))\biggr|
&\le
\sum_{n=1}^{N}|\psi(x_{n})-\psi_{H}(x_{n})|\\
&\le
\frac{1}{2H}\sum_{|h|<H}\biggl(1-\frac{|h|}{H}\biggr)
\sum_{n=1}^{N}e(hx_{n})
\ll
\frac{1}{H}\sum_{0<|h|<H}
\biggl|\sum_{n=1}^{N}e(hx_{n})\biggr|
+
\frac{N}{H}.
\end{align}
Since $0<|h|<H$ in the above sum and
\[
\sum_{n=1}^{N}e(-hx_{n})
=
\overline{\sum_{n=1}^{N}e(hx_{n})},
\]
we have
\[
\biggl|\sum_{n=1}^{N}(\psi(x_{n})-\psi_{H}(x_{n}))\biggr|
\ll
\sum_{1\le h\le H}\frac{1}{h}
\biggl|\sum_{n=1}^{N}e(hx_{n})\biggr|
+
\frac{N}{H}.
\]
Also, we have
\begin{align}
\biggl|
\sum_{n=1}^{N}
\psi_{H}(x_{n})
\biggr|
=
\biggl|\sum_{1\le|h|<H}\frac{c_{H}(h)}{2\pi i h}\sum_{n=1}^{N}e(hx_{n})\biggr|
\ll
\sum_{1\le h\le H}\frac{1}{h}\biggl|\sum_{n=1}^{N}e(hx_{n})\biggr|
\end{align}
by using the bound \cref{lem:Vaaler:c_bound} of \cref{lem:Vaaler}.
By combining the above results, we obtain the assertion.
\end{proof}

We now quickly recall the definition of exponent pair~\cite[Chapter~3]{GrahamKolesnik}

\begin{definition}[Model phase function]
Let $y$, $s$ be positive real numbers.
We define \textsl{the model phase function $\phi(x)=\phi_{y,s}(x)$ of amplitude $y$ and index $s$} by
\[
\phi(x)
=\phi_{y,s}(x)
\coloneqq
\left\{
\begin{array}{ll}
\displaystyle
\frac{yx^{1-s}}{1-s}&(\text{if $s\neq1$}),\\[3mm]
y\log x&(\text{if $s=1$}).
\end{array}
\right.
\]
\end{definition}

\begin{definition}[Admissible phase function]
Let $N$, $y$, $s$, $\epsilon$ be positive real numbers
and $P$ be a positive integer.
A real-valued function $f$ defined on $[a,b]$
is said to be
\textsl{an admissble phase function of scale $N$, order $P$, amplitude $y$, index $s$ and precision $\epsilon$}
if
\begin{enumerate}[label=\textup{(\roman*)}]
\item $[a,b]\subset[N,2N]$,
\item $f$ is $P$-times continuously differentiable on $[a,b]$,
\item for any integer $p$ with $1\le p\le P$, we have
\[
|f^{(p)}(x)-\phi_{y,s}^{(p)}(x)|
\le
\epsilon\cdot|\phi_{y,s}^{(p)}(x)|,
\]
where $\phi_{y,s}(x)$ is the model phase function of amplitude $y$ and index $s$.
\end{enumerate}
We denote the set of all admissble phase functions
of scale $N$, order $P$, amplitude $y$, index $s$ and precision $\epsilon$
by $\mathscr{F}(N,P,y,s,\epsilon)$.
\end{definition}

\begin{definition}[Exponent pair]
\label{def:exp_pair}
A pair of real numbers $(k,\ell)$ with
\begin{equation}
\label{exp_pair:range}
0\le k\le\tfrac{1}{2}\le\ell\le1
\end{equation}
is said to be \textsl{an exponent pair} if for any positive real number $s$,
there are a positive integer $P=P(k,\ell,s)$ and a positive real number $\epsilon=\epsilon(k,\ell,s)$
such that for any positive real numbers $N, y, s$
and any admissible function $f\in\mathscr{F}(N,P,y,s,\epsilon)$ defined on $[a,b]$, we have
\begin{equation}
\label{exp_pair:bound}
\sum_{a<n\le b}e(f(n))
\ll
L^{k}N^{\ell}+L^{-1},
\end{equation}
where $L\coloneqq yN^{-s}$ and the implicit constant depends only on $k,\ell,s$.
\end{definition}

\begin{lemma}
\label{lem:exp_sum}
For the sum
\[
\mathfrak{S}_{\delta,\eta}(N,N')
\coloneqq
\sum_{N<n\le N'}
\psi\biggl(\frac{x}{qn+\delta}-\eta\biggr)
\]
with
\[
x\ge1,\quad
q\in\mathbb{N},\quad
N\ge 1,\quad
N\le N'\le 2N,\quad
0\le\delta\le q+1
\and
\eta\in\mathbb{R}
\]
and an exponent pair $(k,\ell)$, we have
\[
\mathfrak{S}_{\delta,\eta}(N,N')
\ll
(x/q)^{\frac{k}{k+1}}
N^{\frac{\ell-k}{k+1}}
+
(x/q)^{-1}N^{2},
\]
where the implicit constant depends only on $k,\ell$.
\end{lemma}
\begin{proof}
We may assume that $N$ is sufficiently large in terms of $k,\ell$.
Also, if $k=0$, then we should have $\ell=1$ as shown in the second last paragraph in p.~31 of \cite{GrahamKolesnik}
and so the assertion is trivial since $(x/q)^{\frac{k}{k+1}}N^{\frac{\ell-k}{k+1}}=N$ if $(k,\ell)=(0,1)$.
We may thus also assume $k>0$.
By \cref{lem:Erdos_Turan}, we have
\begin{align}
\mathfrak{S}_{\delta,\eta}(N,N')
&\ll
\sum_{1\le h\le H}
\frac{1}{h}
\biggl|
\sum_{N<n\le N'}
e\biggl(\frac{hx}{qn+\delta}-h\eta\biggr)
\biggr|
+
\frac{N}{H}\\
&\ll
\sum_{1\le h\le H}
\frac{1}{h}
\biggl|
\sum_{N<n\le N'}
e\biggl(-\frac{hx}{qn+\delta}\biggr)
\biggr|
+
\frac{N}{H}.
\end{align}
Let us consider the phase function
\[
f
\colon
[N,N']
\to
\mathbb{R}
\semicolon
u\mapsto
-\frac{hx}{qu+\delta}.
\]
For $p\in\mathbb{Z}_{\ge1}$ and $u\in[N,N']$, we have
\begin{align}
f^{(p)}(u)
=
(-1)^{p-1}p!
\biggl(\frac{hx}{q}\biggr)
\frac{1}{(u+\frac{\delta}{q})^{p+1}}
&=
(-1)^{p-1}p!
\biggl(\frac{hx}{qu^{p+1}}\biggr)
\biggl(
1+O_{p}\biggl(\frac{\delta}{qu}\biggr)
\biggr)\\
&=
\phi_{\frac{hx}{q},2}^{(p)}(u)
\biggl(
1+O_{p}\biggl(\frac{1}{u}\biggr)
\biggr)
\end{align}
since $0\le\delta\le q$. Thus, if $N$ is sufficiently large in terms of $(k,\ell)$, we have
\[
f\in\mathscr{F}(N,P,\tfrac{hx}{q},2,\epsilon)
\]
with $P=P(k,\ell,2)$ and $\epsilon=\epsilon(k,\ell,2)$
given as in \cref{def:exp_pair}. This gives
\[
\sum_{N<n\le N'}
e\biggl(-\frac{hx}{dn+\delta}\biggr)
\ll
\biggl(\frac{hx}{qN^{2}}\biggr)^{k}N^{\ell}
+
\frac{qN^{2}}{hx}
=
\biggl(\frac{hx}{q}\biggr)^{k}N^{\ell-2k}
+
\frac{qN^{2}}{hx}.
\]
Therefore, since we are assuming $k>0$, we have
\begin{align}
\mathfrak{S}_{\delta,\eta}(N,N')
\ll
\biggl(\frac{Hx}{q}\biggr)^{k}N^{\ell-2k}
+
\frac{N^{2}}{x/q}
+
\frac{N}{H}.
\end{align}
Since
\[
\biggl(\biggl(\frac{Hx}{q}\biggr)^{k}N^{\ell-2k}\biggr)^{\frac{1}{k+1}}
\biggl(\frac{N}{H}\biggr)^{\frac{k}{k+1}}
=
\biggl(\frac{x}{q}\biggr)^{\frac{k}{k+1}}
N^{\frac{\ell-k}{k+1}},
\]
by optimizing $H>0$, we obtain the lemma.
\end{proof}

We can then bound the sum with $\psi$ as follows:

\begin{lemma}
\label{lem:R_bound}
For an exponent pair $(k,\ell)$, $x\ge1$, $a,q\in\mathbb{Z}$ with $1\le a\le q$ and $q\ge1$,
\[
R_{1},R_{1}^{\ast},R_{2}
\ll
\biggl(\frac{x}{q}\biggr)^{\frac{k+\ell}{2k+2}}
+
1,
\]
where $R_{1},R_{1}^{\ast},R_{2}$ are as in \cref{prop:S_d_decomp,prop:S_d_ast_decomp}
and the implicit constant depends only on $(k,\ell)$.
\end{lemma}
\begin{proof}
Since the exponent pair $(k,\ell)=BA^{3}B(0,1)=(\frac{11}{30},\frac{16}{30})$ satisfies
\[
\frac{k+\ell}{2k+2}\biggr\vert_{(k,\ell)=(\frac{11}{30},\frac{16}{30})}
=
\frac{27}{82}
<
\frac{1}{3}
=
\frac{k+\ell}{2k+2}\biggr\vert_{(k,\ell)=(\frac{1}{2},\frac{1}{2})},
\]
we may assume $(k,\ell)\neq(\frac{1}{2},\frac{1}{2})$ and so $k<\ell$.
When $\frac{x}{q}\ll1$, we obviously have
$R_{1},R_{1}^{\ast},R_{2}\ll\sqrt{x/q}+1\ll1$.
Thus, we may assume $x/q$ is sufficiently large.
By decomposing dyadically, we have
\[
R_{1},R_{1}^{\ast},R_{2}
\ll
\sum_{\substack{m\ge 0\\2^{m}\le\sqrt{x/q}}}
\sup_{\substack{\delta\in[0,q+1]\\\eta\in\mathbb{R}}}
|\mathfrak{S}_{\delta,\eta}(2^{m},\min(2^{m+1},\sqrt{x/q}))|
+
1
\]
By \cref{lem:exp_sum}, we have
\begin{align}
R_{1},R_{1}^{\ast},R_{2}
\ll
\sum_{\substack{m\ge 0\\2^{m}\le\sqrt{x/q}}}
\biggl(
\biggl(\frac{x}{q}\biggr)^{\frac{k}{k+1}}
(2^{m})^{\frac{\ell-k}{k+1}}
+
(x/q)^{-1}(2^{m})^{2}
\biggr)
+1
\ll
\biggl(\frac{x}{q}\biggr)^{\frac{k+\ell}{2k+2}}
+1
\end{align}
since $k<\ell$. This completes the proof.
\end{proof}

Now \cref{thm:WuYu_improved} is just a corollary of the above results.

\begin{proof}[Proof of \cref{thm:WuYu_improved}]
Since the left-hand sides of the assertions of \cref{thm:WuYu_improved} are invariant under the shift $a\leadsto a+q$,
we may assume $1\le a\le q$.
We can then use \cref{prop:S_d_decomp,prop:S_d_ast_decomp}
and \cref{lem:R_bound} to obtain the assertion.
\end{proof}

We shall prepare the special case of \cref{thm:WuYu_improved}
to use in the next section.

\begin{corollary}
\label{cor:YuWu}
For $x\ge1$, $a,q\in\mathbb{Z}$ with $1\le q\le x$, we have
\begin{align}
S(x,q,a)
&=
x\sum_{\substack{n=1\\n\equiv a\ \mod{q}}}^{\infty}\frac{1}{n(n+1)}
+
O\biggl(
\biggl(\frac{x}{q}\biggr)^{\frac{1}{3}}
\biggr),\\
S^{\ast}(x,q,a)
&=
\frac{2\sqrt{x}}{q}
+
O\biggl(
\biggl(\frac{x}{q}\biggr)^{\frac{1}{3}}
\biggr),
\end{align}
where the implicit constants are absolute.
\end{corollary}
\begin{proof}
This follows by \cref{thm:WuYu_improved} with $(k,\ell)=(\frac{1}{2},\frac{1}{2})$.
(As it can be seen from the proof of \cref{thm:WuYu_improved},
the actually used exponent pair is $(k,\ell)=(\frac{11}{30},\frac{16}{30})$
to avoid the singular case $k=\ell$.)
\end{proof}

\section{Primitive lattice points and the \texorpdfstring{$[x/n]$}{[x/n]}-sequence}
\label{sec:primitive_lattice_points}
In this section, we shall prove \cref{thm:primitive_Sx} and \cref{thm:primitive_xn}.

%
%

\begin{proof}[Proof of \cref{thm:primitive_Sx}]
Since $\ell\le1$, we have $\frac{k+\ell}{k+1}\le1$.
When $\frac{k+\ell}{k+1}=1$, then the error term is of the size $O(x)$
and so the assertion is trivial since the left-hand side can be bounded by $(\#\mathscr{S}(x))^{2}\ll x$.
We may thus assume $\frac{k+\ell}{k+1}<1$.
By using the M\"{o}bius function, we have
\begin{align}
S
\coloneqq
\sum_{\substack{
m,n\in\mathscr{S}(x)\\
(m,n)=1
}}
1
=
\sum_{\substack{
m,n\in\mathscr{S}(x)
}}
\sum_{d\mid(m,n)}\mu(d)
=
\sum_{d\le x}\mu(d)
\biggl(
\sum_{\substack{
n\in\mathscr{S}(x)\\
d\mid n
}}
1
\biggr)^{2}
=
\sum_{d\le x}\mu(d)
S^{\ast}(x,d,0)^{2}.
\end{align}
Take a parameter $U\in[1,x]$ and decompose the above sum as
\[
S
=
\sum_{d\le U}
+
\sum_{d>U}
\eqqcolon
S_{1}+S_{2},
\quad\text{say}.
\]
For $S_{1}$, since we are assuming $0<\frac{k+\ell}{k+1}<1$, we can use \cref{thm:WuYu_improved} to get
\begin{equation}
\label{thm:primitive_Sx:S1}
\begin{aligned}
S_{1}
&=
\sum_{d\le U}
\mu(d)
\biggl(\frac{4x}{d^{2}}
+
O\biggl(\frac{x^{\frac{1}{2}+\frac{k+\ell}{2k+2}}}{d^{1+\frac{k+\ell}{2k+2}}}
+
\biggl(\frac{x}{d}\biggr)^{\frac{k+\ell}{k+1}}
\biggr)
\biggr)\\
&=
4x\sum_{d\le U}
\frac{\mu(d)}{d^{2}}
+
O\bigl(
x^{\frac{1}{2}+\frac{k+\ell}{2k+2}}
+
x^{\frac{k+\ell}{k+1}}U^{1-\frac{k+\ell}{k+1}}
\bigr)\\
&=
\frac{4}{\zeta(2)}x
+
O\bigl(
xU^{-1}
+
x^{\frac{1}{2}+\frac{k+\ell}{2k+2}}
+
x^{\frac{k+\ell}{k+1}}U^{1-\frac{k+\ell}{k+1}}
\bigr).
\end{aligned}
\end{equation}
For the sum $S_{2}$, we apply \cref{thm:WuYu_improved}
to only one factor $S(x,d,0)$ to get
\begin{align}
\label{thm:primitive_Sx:S2_decomp}
S_{2}
\ll
\sum_{U<d\le x}
\biggl(\frac{x^{\frac{1}{2}}}{d}+\biggl(\frac{x}{d}\biggr)^{\frac{k+\ell}{2k+2}}\biggr)
\sum_{
\substack{n\in\mathscr{S}(x)\\
d\mid n
}}
1
\ll
\biggl(\frac{x^{\frac{1}{2}}}{U}+\biggl(\frac{x}{U}\biggr)^{\frac{k+\ell}{2k+2}}\biggr)
\sum_{U<d\le x}
\sum_{
\substack{n\in\mathscr{S}(x)\\
d\mid n
}}
1.
\end{align}
For this sum, by swapping the summation and using the bound
\[
\tau(n)\ll_{\epsilon} n^{\epsilon}
\]
for any $\epsilon>0$, where the implicit constant depends on $\epsilon$, we get
\begin{align}
\sum_{U<d\le x}
\sum_{
\substack{n\in\mathscr{S}(x)\\
d\mid n
}}
1
=
\sum_{n\in\mathscr{S}(x)}
\sum_{\substack{
d\mid n\\
U<d\le x
}}
1
\le
\sum_{n\in\mathscr{S}(x)}
\tau(n)
\ll_{\epsilon}
x^{\epsilon}\#\mathscr{S}(x)
\ll
x^{\frac{1}{2}+\epsilon}
\end{align}
for any $\epsilon>0$, where the implicit constant depends on $\epsilon$.
On inserting this into \cref{thm:primitive_Sx:S2_decomp}, we get
\begin{equation}
\label{thm:primitive_Sx:S2}
S_{2}
\ll_{\epsilon}
x^{1+\epsilon}U^{-1}
+
x^{\frac{1}{2}+\frac{k+\ell}{2k+2}+\epsilon}
U^{-\frac{k+\ell}{2k+2}}.
\end{equation}
By combining \cref{thm:primitive_Sx:S1} and \cref{thm:primitive_Sx:S2}, we obtain
\[
S
=
\frac{4}{\zeta(2)}x
+
O_{\epsilon}\bigl(
x^{\frac{1}{2}+\frac{k+\ell}{2k+2}}
+
x^{\frac{k+\ell}{k+1}}U^{1-\frac{k+\ell}{k+1}}
+
x^{1+\epsilon}U^{-1}
+
x^{\frac{1}{2}+\frac{k+\ell}{2k+2}+\epsilon}
U^{-\frac{k+\ell}{2k+2}}
\bigr).
\]
By taking $U=x^{\frac{1}{2}}\in[1,x]$ and noting $\frac{k+\ell}{k+1}<1$, we have
\[
S
=
\frac{4}{\zeta(2)}x
+
O_{\epsilon}\bigl(
x^{\frac{1}{2}+\frac{k+\ell}{2k+2}}
+
x^{\frac{1}{2}+\frac{k+\ell}{4k+4}+\epsilon}
\bigr).
\]
Finally, by taking
\[
\epsilon=\frac{k+\ell}{4k+4}>0,
\]
we obtain the assertion.
\end{proof}

The proof of \cref{thm:primitive_Sx} is more straightforward.

\begin{proof}[Proof of \cref{thm:primitive_xn}]
Since $(k,\ell)=(\frac{1}{2},\frac{1}{2})$ is an exponent pair
and gives a better result than the exponent pair $(k,\ell)=(0,1)$,
we may assume $(k,\ell)\neq(0,1)$.
We have
\begin{align}
\sum_{\substack{m,n\le x\\([x/m],[x/n])=1}}1
&=
\sum_{m,n\le x}\sum_{d\mid([x/m],[x/n])}\mu(d)
=
\sum_{d\le x}\mu(d)S(x,d,0)^{2}.
\end{align}
By using \cref{thm:WuYu_improved} and
\[
\sum_{\substack{n=1\\d\mid n}}^{\infty}\frac{1}{n(n+1)}
\le
\frac{1}{d^{2}}
\sum_{n=1}^{\infty}\frac{1}{n^{2}}
\ll
\frac{1}{d^{2}}
\]
we have
\[
S(x,d,0)^{2}
=
x^{2}\biggl(\sum_{\substack{n=1\\d\mid n}}^{\infty}\frac{1}{n(n+1)}\biggr)^{2}
+
O\biggl(
\frac{x}{d^{2}}\biggl(\frac{x}{d}\biggr)^{\frac{k+\ell}{2k+2}}
+
\biggl(\frac{x}{d}\biggr)^{\frac{k+\ell}{k+1}}
\biggr).
\]
Since $(k,\ell)\neq(0,1)$ implies $\frac{k+\ell}{k+1}<1$, we have
\begin{align}
\sum_{\substack{m,n\le x\\([x/m],[x/n])=1}}1
=
x^{2}\sum_{\substack{d\le x}}\mu(d)\sum_{\substack{m,n=1\\d\mid(m,n)}}^{\infty}\frac{1}{mn(m+1)(n+1)}
+
O(x^{1+\frac{k+\ell}{2k+2}}).
\end{align}
We now have
\begin{align}
\sum_{\substack{d\le x}}\mu(d)\sum_{\substack{m,n=1\\d\mid(m,n)}}^{\infty}\frac{1}{mn(m+1)(n+1)}
&=
\sum_{m,n=1}^{\infty}\frac{1}{mn(m+1)(n+1)}
\sum_{\substack{d\mid (m,n)\\d\le x}}\mu(d)\\
&=
\mathfrak{S}
+
O\biggl(\sum_{\substack{m,n=1\\(m,n)>x}}^{\infty}\frac{\tau((m,n))}{mn(m+1)(n+1)}\biggr).
\end{align}
The last error term is bounded as
\begin{align}
\sum_{\substack{m,n=1\\(m,n)>x}}^{\infty}\frac{\tau((m,n))}{mn(m+1)(n+1)}
&\ll
\sum_{\substack{m,n=1\\(m,n)>x}}^{\infty}\frac{\tau((m,n))}{m^{2}n^{2}}
\ll
\sum_{d>x}\frac{\tau(d)}{d^{4}}
\biggl(\sum_{m=1}^{\infty}\frac{1}{m^{2}}\biggr)^{2}
\ll
\frac{\log x}{x^{3}}.
\end{align}
We thus obtain the theorem.
\end{proof}

\section{Mean value of small arithmetic functions over the \texorpdfstring{$[x/p]$}{[x/p]}-sequence}
\label{sec:xp_small}
In this section, we deal with the multiplicative Titchmarsh divisor theorem
with small arithmetic function, i.e.\ \cref{thm:MTD_easy_with_vonMangoldt}
and \cref{thm:MTD_easy_unweighted}.

The proof of \cref{thm:MTD_easy_unweighted} is straightforward
and essentially follows \cite[Subsection~3.1]{Wu_general}.

\begin{proof}[Proof of \cref{thm:MTD_easy_with_vonMangoldt}]
Write
\[
S
\coloneqq
\sum_{n\le x}\Lambda(n)a\biggl(\biggl[\frac{x}{n}\biggr]\biggr)
=
\sum_{n\le x}a(n)\sum_{\frac{x}{n+1}<m\le\frac{x}{n}}\Lambda(m)
=
\sum_{n\le x^{\frac{1}{2}}}
+
\sum_{x^{\frac{1}{2}}<n\le x}
\eqqcolon
{\sum}_{1}+{\sum}_{2},
\quad\text{say}.
\]
By the prime number theorem and \cref{thm:MTD_easy_with_vonMangoldt:a_mean} in \cref{thm:MTD_easy_with_vonMangoldt},
we have
\begin{align}
{\sum}_{1}
&=
x\sum_{n\le x^{\frac{1}{2}}}\frac{a(n)}{n(n+1)}
+
O\biggl(
x\sum_{n\le x^{\frac{1}{2}}}\frac{|a(n)|}{n}\exp\biggl(-2c\sqrt{\log\frac{x}{n}}\biggr)
\biggr)\\
&=
x\sum_{n=1}^{\infty}\frac{a(n)}{n(n+1)}
+
O(x\exp(-c\sqrt{\log x}))
\end{align}
with some absolute constant $c>0$.
On the other hand, by \cref{thm:MTD_easy_with_vonMangoldt:sup} in \cref{thm:MTD_easy_with_vonMangoldt}, we have
\begin{align}
{\sum}_{2}
&\ll
x^{\theta}
\sum_{x^{\frac{1}{2}}<n\le x}\sum_{\frac{x}{n+1}<m\le\frac{x}{n}}\Lambda(m)
=
x^{\theta}
\sum_{m\le x^{\frac{1}{2}}}\Lambda(m)\sum_{\frac{x}{m}-1<n\le\frac{x}{m}}1
\ll
x^{\theta+\frac{1}{2}}
\ll
x\exp(-c\sqrt{\log x}).
\end{align}
This completes the proof.
\end{proof}

The proof of \cref{thm:MTD_easy_unweighted} is also easy
except the evaluation of the main term is somehow involved.

\begin{proof}[Proof of \cref{thm:MTD_easy_unweighted}]
We may assume $x$ is sufficiently large.
We begin with
\begin{align}
\sum_{p\le x}a\biggl(\biggl[\frac{x}{p}\biggr]\biggr)
=
\sum_{n\le x}a(n)\sum_{\substack{\frac{x}{n+1}<p\le\frac{x}{n}}}1
=
\sum_{n\le x^{\frac{1}{2}}}
+
\sum_{n> x^{\frac{1}{2}}}
=
{\sum}_{1}+{\sum}_{2},\quad\text{say}.
\end{align}
For the sum ${\sum}_{1}$, by the prime number theorem and \cref{thm:MTD_easy_with_vonMangoldt:a_mean}, we have
\begin{equation}
\label{thm:MTD_easy_unweighted:Sigma1_M}
\begin{aligned}
{\sum}_{1}
&=
\sum_{n\le x^{\frac{1}{2}}}
a(n)\int_{\frac{x}{n+1}}^{\frac{x}{n}}\frac{du}{\log u}
+
O\biggl(
x\sum_{n\le x^{\frac{1}{2}}}\frac{|a(n)|}{n}\exp\biggl(-2c\sqrt{\log\frac{x}{n}}\biggr)
\biggr)\\
&=
\sum_{n\le x^{\frac{1}{2}}}
a(n)\int_{\frac{x}{n+1}}^{\frac{x}{n}}\frac{du}{\log u}
+
O(x\exp(-c\sqrt{\log x}))\\
&=
M
+
O(x\exp(-c\sqrt{\log x})),\quad\text{say}.
\end{aligned}
\end{equation}
Write $y\coloneqq x^{\frac{1}{2}}$.
By swapping the summation and integration and using \cref{thm:MTD_easy_with_vonMangoldt:sup}, we get
\begin{align}
M
&=
\int_{\frac{x}{y+1}}^{x}
\biggl(\sum_{\frac{x}{u}-1<n\le\min(y,\frac{x}{u})}a(n)\biggr)
\frac{du}{\log u}
=
\int_{\frac{x}{y}}^{x}
\biggl(\sum_{\frac{x}{u}-1<n\le\frac{x}{u}}a(n)\biggr)
\frac{du}{\log u}
+O\biggl(
x^{\frac{\theta}{2}}
\int_{\frac{x}{y+1}}^{\frac{x}{y}}\frac{du}{\log u}
\biggr).
\end{align}
We can bound the error term as
\[
x^{\frac{\theta}{2}}
\int_{\frac{x}{y+1}}^{\frac{x}{y}}\frac{du}{\log u}
\ll
\frac{x^{\frac{\theta}{2}}}{\log x}\cdot\frac{x}{y^{2}}
\ll
\frac{x^{\frac{\theta}{2}}}{\log x}
\ll
x^{\frac{1}{4}}
\]
and so
\[
M
=
\int_{\frac{x}{y}}^{x}
\biggl(\sum_{\frac{x}{u}-1<n\le\frac{x}{u}}a(n)\biggr)
\frac{du}{\log u}
+
O(x^{\frac{1}{4}}).
\]
By changing variables via $u\leadsto x/u$, we have
\begin{align}
\int_{\frac{x}{y}}^{x}
\biggl(\sum_{\frac{x}{u}-1<n\le\frac{x}{u}}a(n)\biggr)
\frac{du}{\log u}
&=
x
\int_{1}^{y}
\biggl(\sum_{u-1<n\le u}a(n)\biggr)
\frac{du}{u^2\log\frac{x}{u}}\\
&=
\frac{x}{\log x}
\int_{1}^{y}\frac{a([u])}{u^2}
\frac{du}{1-\frac{\log u}{\log x}}\\
&=
\frac{x}{\log x}
\sum_{k=0}^{K}\biggl(\frac{1}{\log x}\biggr)^{k}
\int_{1}^{y}\frac{a([u])(\log u)^{k}}{u^2}du\\
&\hspace{10mm}+
\frac{x}{\log x}\biggl(\frac{1}{\log x}\biggr)^{K+1}
\int_{1}^{y}\frac{a([u])(\log u)^{K+1}}{u^2}\frac{du}{1-\frac{\log u}{\log x}}.
\end{align}
The assumption \cref{thm:MTD_easy_with_vonMangoldt:sup} implies the bounds
\begin{align}
\int_{y}^{\infty}\frac{a([u])(\log u)^{k}}{u^2}du
\ll
\int_{y}^{\infty}\frac{(\log u)^{k}}{u^{\frac{3}{2}}}du
\ll_{k}
\frac{(\log x)^{k}}{x^{\frac{1}{4}}}
\end{align}
and
\begin{align}
\int_{1}^{y}\frac{a([u])(\log u)^{K+1}}{u^2}\frac{du}{1-\frac{\log u}{\log x}}
\ll
\int_{1}^{y}\frac{(\log u)^{K+1}}{u^{\frac{3}{2}}}du
\ll_{K}
1.
\end{align}
Therefore, we have
\[
M
=
\frac{x}{\log x}
\sum_{k=0}^{K}\biggl(\frac{1}{\log x}\biggr)^{k}
\int_{1}^{\infty}\frac{a([u])(\log u)^{k}}{u^2}du
+
O_{K}\biggl(\frac{x}{\log x}\biggl(\frac{1}{\log x}\biggr)^{K+1}\biggr).
\]
By substituting this formula into \cref{thm:MTD_easy_unweighted:Sigma1_M},
we obtain
\begin{align}
{\sum}_{1}
=
\frac{x}{\log x}
\sum_{k=0}^{K}\biggl(\frac{1}{\log x}\biggr)^{k}
\int_{1}^{\infty}\frac{a([u])(\log u)^{k}}{u^2}du
+
O_{K}\biggl(\frac{x}{\log x}\biggl(\frac{1}{\log x}\biggr)^{K+1}\biggr).
\end{align}
For the sum ${\sum}_{2}$, the assumption \cref{thm:MTD_easy_with_vonMangoldt:sup} yields
\begin{align}
{\sum}_{2}
\ll
x^{\theta}
\sum_{x^{\frac{1}{2}}<n\le x}\sum_{\substack{\frac{x}{n+1}<p\le\frac{x}{n}}}1
\le
x^{\theta}
\sum_{p\le x^{\frac{1}{2}}}\sum_{x/p-1<n\le x/p}1
\ll
x^{\frac{1}{2}+\theta}
\ll_{K}
\frac{x}{\log x}\biggl(\frac{1}{\log x}\biggr)^{K+1},
\end{align}
which is enough for our purpose.
By combining the above results, we obtain the result.
\end{proof}


\section{Multiplicative Titchmarsh divisor problem with \texorpdfstring{$\phi$}{phi}}
\label{sec:xp_phi}
In this section, we consider the multiplicative Titchmarsh divisor problem
with the Euler totient $\phi(n)$.

We first recall bounds for exponential sums with arithmetic function
of the Korobov--Vinogradov--Walfisz type:

\begin{lemma}
\label{lem:expsum_KVW}
Let $v(n)=\mu(n),\Lambda(n)$ or $\mathbbm{1}_{n:\textup{prime}}$.
For $A\ge1$, $P,P',Q\ge4$ with $P\le P'\le 2P$ and $\beta\in\mathbb{R}$,
there exists $B=B(A)\ge1$ such that
\[
\sum_{P<n\le P'}v(n)\psi\biggl(\frac{Q}{n}+\eta\biggr)
\ll
P(\log Q)^{-A}
\]
provided
\[
\exp(B(\log Q)^{\frac{2}{3}}(\log\log Q)^{\frac{1}{3}})
\le
P
\le
Q(\log Q)^{-B},
\]
where the implicit constant depends only on $A$.
\end{lemma}
\begin{proof}
This can be proven by the exponential sum estimate given as Lemma~6, 7, 12 of \cite{Yuta:Walfisz}
combined with \cref{lem:Erdos_Turan}.
Note that the effect of the shift $\eta$ can be removed after the application of \cref{lem:Erdos_Turan}.
The case $a(n)=\mu(n)$ is essentially due to H. Q. Liu~\cite{HQLiu}.
\end{proof}

We also recall the usual expansion of the error term
of the summatory function of the Euler totient function:

\begin{lemma}
\label{lem:Euler_phi_remainder}
For $x\ge 4$, $\delta>0$ and $u\in[x^{\delta},x]$, we have
\[
\sum_{n\le u}\phi(n)
=
\frac{1}{2\zeta(2)}u^{2}+E(u)+O(ue^{-c\sqrt{\log x}})
\]
with some constant $c=c(\delta)>0$ depending on $\delta$ and
\[
E(u)
=
-u\sum_{d\le\frac{u}{y}}
\frac{\mu(d)}{d}
\psi\biggl(\frac{u}{d}\biggr),\quad
y\coloneqq\exp(c\sqrt{\log x})
\]
and the implicit constant depends only on $\delta$.
\end{lemma}
\begin{proof}
This can be proven by the same method as the proof of Lemma~2.2 of \cite{Zhai}.
\end{proof}

To prove \cref{thm:phi_xn_Titchmarsh_vonMangoldt},
we need Huxley's prime number theorem in short intervals:

\begin{lemma}
\label{lem:Huxley_PNT}
For real numbers $X,H,\epsilon,A$ with $4\le H\le X$, $\epsilon>0$ and $A\ge1$, we have
\[
\sum_{X<n\le X+H}\Lambda(n)
=
H+O(H(\log X)^{-A})
\]
provided $X^{\frac{7}{12}+\epsilon}\le H\le X$,
where the implicit constant depends on $\epsilon$ and $A$.
\end{lemma}
\begin{proof}
This can be proven by the standard argument
using Huxley's zero density estimate~\cite{Huxley:invent}.
\end{proof}

\begin{proof}[Proof of \cref{thm:phi_xn_Titchmarsh_vonMangoldt}]
For a large constant $B\ge1$, let
\[
R\coloneqq\exp(B(\log x)^{\frac{2}{3}}(\log\log x)^{\frac{1}{3}}).
\]
We then decompose the sum as
\[
\sum_{n\le x}\Lambda(n)\phi\biggl(\biggl[\frac{x}{n}\biggr]\biggr)
=
\sum_{n\le R}
+
\sum_{R<n\le x^{\frac{3}{4}}}
+
\sum_{x^{\frac{3}{4}}<n\le x}
=
{\sum}_{1}+{\sum}_{2}+{\sum}_{3},\quad\text{say}.
\]

For the sum ${\sum}_{1}$, by using $1\le\phi(n)\le n$, we just estimate as
\[
{\sum}_{1}
=
\sum_{n\le R}\Lambda(n)\phi\biggl(\biggl[\frac{x}{n}\biggr]\biggr)
\le
x\sum_{n\le R}\frac{\Lambda(n)}{n}
\ll
x\log R
\ll
x(\log x)^{\frac{2}{3}}(\log\log x)^{\frac{1}{3}}.
\]

For the sum ${\sum}_{2}$,
we first rewrite it as
\begin{align}
{\sum}_{2}
=\sum_{R<n\le x^{\frac{3}{4}}}\Lambda(n)\phi\biggl(\biggl[\frac{x}{n}\biggr]\biggr)
&=
\sum_{R<n\le  x^{\frac{3}{4}}}\Lambda(n)
\sum_{\frac{x}{n}-1<m\le\frac{x}{n}}\phi(m)\\
&=
\sum_{R<n\le  x^{\frac{3}{4}}}\Lambda(n)
\biggl(
\sum_{m\le\frac{x}{n}}\phi(m)
-
\sum_{m\le\frac{x}{n}-1}\phi(m)
\biggr)
\end{align}
for large $x$. By using the notation of \cref{lem:Euler_phi_remainder}
and noting that
\[
\frac{1}{2\zeta(2)}\biggl(\frac{x}{n}\biggr)^{2}
-
\frac{1}{2\zeta(2)}\biggl(\frac{x}{n}-1\biggr)^{2}
=
\frac{1}{\zeta(2)}\frac{x}{n}-\frac{1}{2\zeta(2)},
\]
we have
\begin{align}
{\sum}_{2}
&=
\frac{1}{\zeta(2)}x\sum_{R<n\le x^{\frac{3}{4}}}\frac{\Lambda(n)}{n}
+E^{(0)}-E^{(1)}
+O\biggl(
\sum_{n\le x^{\frac{3}{4}}}\Lambda(n)
+
x\exp(-c\sqrt{\log x})
\sum_{n\le x^{\frac{3}{4}}}
\frac{\Lambda(n)}{n}
\biggr)\\
&=
\frac{3}{4\zeta(2)}x\log x
+E^{(0)}-E^{(1)}+O(x(\log x)^{\frac{2}{3}}(\log\log x)^{\frac{1}{3}}),
\end{align}
where
\begin{align}
E^{(\delta)}
\coloneqq
\sum_{R<n\le x^{\frac{3}{4}}}\Lambda(n)E\biggl(\frac{x}{n}-\delta\biggr)
\end{align}
for $\delta=0,1$.

For the sum ${\sum}_{3}$, we have
\begin{align}
{\sum}_{3}
&=
\sum_{x^{\frac{3}{4}}<n\le x}\Lambda(n)\phi\biggl(\biggl[\frac{x}{n}\biggr]\biggr)
=
\sum_{m\le x^{\frac{1}{4}}}\phi(m)
\sum_{\substack{x^{\frac{3}{4}}<n\le x\\\frac{x}{m+1}<n\le\frac{x}{m}}}\Lambda(n)\\
&=
\sum_{m\le x^{\frac{1}{4}}}\phi(m)
\sum_{\frac{x}{m+1}<n\le\frac{x}{m}}\Lambda(n)
+
O\biggl(
\sum_{x^{\frac{1}{4}}-1<m\le x^{\frac{1}{4}}}
\phi(m)
\sum_{\frac{x}{m+1}<n\le\frac{x}{m}}\Lambda(n)
\biggr).
\end{align}
Since
\[
\sum_{x^{\frac{1}{4}}-1<m\le x^{\frac{1}{4}}}\phi(m)
\sum_{\frac{x}{m+1}<n\le\frac{x}{m}}\Lambda(n)
\ll
x\log x\sum_{x^{\frac{1}{4}}-1<m\le x^{\frac{1}{4}}}\frac{\phi(m)}{m^{2}}
\ll
x^{\frac{3}{4}}\log x
\ll
x,
\]
we have
\[
{\sum}_{3}
=
\sum_{m\le x^{\frac{1}{4}}}\phi(m)
\sum_{\frac{x}{m+1}<n\le\frac{x}{m}}\Lambda(n)
+
O(x).
\]
By using \cref{lem:Huxley_PNT} and noting that
$m\le x^{\frac{1}{4}}$ implies $\frac{x}{m(m+1)}\ge(\frac{x}{m+1})^{\frac{3}{5}}$
for large $x$, we have
\begin{align}
{\sum}_{3}
&=
x\sum_{m\le x^{\frac{1}{4}}}\frac{\phi(m)}{m(m+1)}
+
O\biggl(
x(\log x)^{-1}\sum_{m\le x^{\frac{1}{4}}}\frac{\phi(m)}{m^2}
+
x\biggr)\\
&=
x\sum_{m\le x^{\frac{1}{4}}}\frac{\phi(m)}{m(m+1)}
+
O(x)
=
x\sum_{m\le x^{\frac{1}{4}}}\frac{\phi(m)}{m^{2}}
+
O(x).
\end{align}
By using
\begin{align}
\sum_{n\le z}\frac{\phi(n)}{n^{2}}
&=
\sum_{n\le z}\frac{\phi(n)}{n}\int_{n}^{z}\frac{du}{u^{2}}
+
\frac{1}{z}\sum_{n\le z}\frac{\phi(n)}{n}\\
&=
\int_{1}^{z}\biggl(\sum_{n\le u}\frac{\phi(n)}{n}\biggr)\frac{du}{u^{2}}
+
O(1)
=
\frac{1}{\zeta(2)}\int_{1}^{z}\frac{du}{u}
+O(1)
=
\frac{1}{\zeta(2)}\log z+O(1),
\end{align}
we have
\[
{\sum}_{3}
=
\frac{1}{4\zeta(2)}x\log x
+
O(x).
\]

By combining the above results, we get
\[
\sum_{n\le x}\Lambda(n)\phi\biggl(\biggl[\frac{x}{n}\biggr]\biggr)
=
\frac{1}{\zeta(2)}x\log x+E^{(0)}-E^{(1)}+O(x(\log x)^{\frac{2}{3}}(\log\log x)^{\frac{1}{3}}).
\]
Thus, it suffices to prove $E^{(0)},E^{(1)}\ll x$.

We then estimate $E^{(\delta)}$.
By substituting the definition of $E(x)$ in \cref{lem:Euler_phi_remainder}, we have
\begin{align}
E^{(\delta)}
&=
\sum_{R<n\le x^{\frac{3}{4}}}\Lambda(n)E\biggl(\frac{x}{n}-\delta\biggr)\\
&=
-x\sum_{\substack{dn\le\frac{x}{y}\\R<n\le x^{\frac{3}{4}}}}
\frac{\mu(d)}{d}\frac{\Lambda(n)}{n}\psi\biggl(\frac{x}{dn}-\frac{\delta}{d}\biggr)\\
&\hspace{20mm}
+O\biggl(
\sum_{R<n\le x^{\frac{3}{4}}}\Lambda(n)
\sum_{d\le\frac{x/n}{y}}\frac{1}{d}
+
y\sum_{R<n\le x^{\frac{3}{4}}}\Lambda(n)
\biggr)\\
&=
-x\sum_{\substack{dn\le\frac{x}{y}\\R<n\le x^{\frac{3}{4}}}}
\frac{\mu(d)}{d}\frac{\Lambda(n)}{n}\psi\biggl(\frac{x}{dn}-\frac{\delta}{d}\biggr)
+
O(x^{\frac{3}{4}}y)
\end{align}
with $y\coloneqq\exp(c\sqrt{\log x})$.
We shall prove $E^{(\delta)}\ll x$. Let us write
\[
E^{(\delta)}
=
\sum_{d\le x^{\frac{1}{2}}}
+
\sum_{d>x^{\frac{1}{2}}}
+
O(x)
=
E_{1}^{(\delta)}+E_{2}^{(\delta)}
+
O(x).
\]
For $E_{1}^{(\delta)}$, we write
\begin{align}
E_{1}^{(\delta)}
=
-x
\sum_{d\le x^{\frac{1}{2}}}\frac{\mu(d)}{d}
\sum_{R<n\le\min(x^{\frac{3}{4}},\frac{x}{yd})}
\frac{\Lambda(n)}{n}\psi\biggl(\frac{\frac{x}{d}}{n}-\frac{\delta}{d}\biggr).
\end{align}
By \cref{lem:expsum_KVW} with noting the condition $n>R$
and using partial summation, we have
\begin{align}
E_{1}^{(\delta)}
\ll
x
\sum_{d\le x^{\frac{1}{2}}}\frac{1}{d}\biggl(\log\frac{x}{d}\biggr)^{-1}
\ll
x
(\log x)^{-1}
\sum_{d\le x^{\frac{1}{2}}}\frac{1}{d}
\ll
x.
\end{align}
On the other hand, for $E_{2}^{(\delta)}$, we have
\begin{align}
E_{2}^{(\delta)}
=
-x
\sum_{R<n\le x^{\frac{3}{4}}}
\frac{\Lambda(n)}{n}
\sum_{x^{\frac{1}{2}}<d\le\frac{x}{yn}}\frac{\mu(d)}{d}
\psi\biggl(\frac{\frac{x}{n}-\delta}{d}\biggr).
\end{align}
By \cref{lem:expsum_KVW} with noting the condition $d>x^{\frac{1}{2}}$,
since $\frac{x}{n}-\delta\gg x^{\frac{1}{4}}$ in the above sum, we have
\[
E_{2}^{(\delta)}
\ll
x
\sum_{R<n\le x^{\frac{3}{4}}}
\frac{\Lambda(n)}{n}
\biggl(\log\biggl(\frac{x}{n}-\delta\biggr)\biggr)^{-1}
\ll
x(\log x)^{-1}
\sum_{R<n\le x^{\frac{3}{4}}}
\frac{\Lambda(n)}{n}
\ll
x.
\]
By combining the above estimates,
we get $E^{(\delta)}\ll x$.
This completes the proof.
\end{proof}

\begin{proof}[Proof of \cref{thm:phi_xn_Titchmarsh}]
For a large constant $B\ge1$, let
\[
R\coloneqq\exp(B(\log x)^{\frac{2}{3}}(\log\log x)^{\frac{1}{3}}).
\]
We then decompose the sum as
\[
\sum_{p\le x}\phi\biggl(\biggl[\frac{x}{p}\biggr]\biggr)
=
\sum_{p\le R}
+
\sum_{R<p\le x^{\frac{3}{4}}}
+
\sum_{x^{\frac{3}{4}}<p\le x}
=
{\sum}_{1}
+{\sum}_{2}
+{\sum}_{3},\quad\text{say}.
\]

For the sum ${\sum}_{1}$, by using $1\le\phi(n)\le n$, we just estimate as
\begin{align}
0\le
{\sum}_{1}
=
\sum_{p\le R}\phi\biggl(\biggl[\frac{x}{p}\biggr]\biggr)
\le
x\sum_{p\le R}\frac{1}{p}
=
x\log\log R+O(x)
=
\frac{2}{3}x\log\log x+O(x\log\log\log x).
\end{align}

For the sum ${\sum}_{2}$,
we first rewrite it as
\begin{align}
{\sum}_{2}
=\sum_{R<p\le x^{\frac{3}{4}}}\phi\biggl(\biggl[\frac{x}{p}\biggr]\biggr)
&=
\sum_{R<p\le  x^{\frac{3}{4}}}
\sum_{\frac{x}{p}-1<m\le\frac{x}{p}}\phi(m)\\
&=
\sum_{R<p\le  x^{\frac{3}{4}}}
\biggl(
\sum_{m\le\frac{x}{p}}\phi(m)
-
\sum_{m\le\frac{x}{p}-1}\phi(m)
\biggr)
\end{align}
for large $x$.
By using the notation of \cref{lem:Euler_phi_remainder} and noting that
\[
\frac{1}{2\zeta(2)}\biggl(\frac{x}{p}\biggr)^{2}
-
\frac{1}{2\zeta(2)}\biggl(\frac{x}{p}-1\biggr)^{2}
=
\frac{1}{\zeta(2)}\frac{x}{p}-\frac{1}{2\zeta(2)},
\]
we have
\begin{align}
{\sum}_{2}
&=
\frac{1}{\zeta(2)}x\sum_{R<p\le x^{\frac{3}{4}}}\frac{1}{p}
+R^{(0)}-R^{(1)}
+O\biggl(
x\exp(-c\sqrt{\log x})\sum_{R<p\le  x^{\frac{3}{4}}}\frac{1}{p}
\biggr)\\
&=
\frac{1}{3\zeta(2)}x\log\log x
+R^{(0)}-R^{(1)}+O(x\log\log\log x),
\end{align}
where
\begin{align}
R^{(\delta)}
&\coloneqq
\sum_{R<p\le x^{\frac{3}{4}}}E\biggl(\frac{x}{p}-\delta\biggr)
\end{align}
for $\delta=0,1$.

For the sum ${\sum}_{3}$, we just have
\[
{\sum}_{3}
\le
x
\sum_{x^{\frac{3}{4}}<p\le x}\frac{1}{p}
\ll
x.
\]

By combining the above results, we get
\begin{align}
\sum_{p\le x}\phi\biggl(\biggl[\frac{x}{p}\biggr]\biggr)
&\le
\biggl(\frac{1}{\zeta(2)}+\frac{2}{3}\biggl(1-\frac{1}{\zeta(2)}\biggr)\biggr)x\log\log x
+R^{(0)}-R^{(1)}
+O(x\log\log\log x),\\
\sum_{p\le x}\phi\biggl(\biggl[\frac{x}{p}\biggr]\biggr)
&\ge
\biggl(\frac{1}{\zeta(2)}-\frac{2}{3}\frac{1}{\zeta(2)}\biggr)x\log\log x
+R^{(0)}-R^{(1)}
+O(x\log\log\log x).
\end{align}
We can then get the bounds
\[
R^{(0)},
R^{(1)}
\ll
x
\]
similarly to the proof of \cref{thm:phi_xn_Titchmarsh_vonMangoldt}
and completes the proof.
\end{proof}

\subsection*{Acknowledgments}
The authors would like to thank {Professor Teerapat Srichan}
for helpful discussion.
The authors also express their gratitude to {Professor Jie Wu}
for his encouragement.
This work was supported by JSPS KAKENHI Grant Numbers
JP19K23402,
JP21K13772,
JP22K13900,
JP22J00025,
JP22J00339,
JP22KJ0375,
JP22KJ1621.

\ifendnotes
\newpage
\begingroup
\parindent 0pt
\parskip 2ex
\def\enotesize{\normalsize}
\theendnotes
\endgroup
\fi

\providecommand{\bysame}{\leavevmode\hbox to3em{\hrulefill}\thinspace}
\providecommand{\MR}{\relax\ifhmode\unskip\space\fi MR }
\providecommand{\MRhref}[2]{%
  \href{http://www.ams.org/mathscinet-getitem?mr=#1}{#2}
}
\providecommand{\href}[2]{#2}

\pagebreak

\begin{flushleft}
{\small
\textsc{%
Kota Saito\\
Faculty of Pure and Applied Sciences, University of Tsukuba,\\
1-1-1 Tennodai, Tsukuba, Ibaraki 305-8577, Japan.
}

\textit{E-mail address}: \texttt{saito.kota.gn@u.tsukuba.ac.jp}
}
\bigskip

{\small
\textsc{%
Yuta Suzuki\\
Department of Mathematics, Rikkyo University,\\
3-34-1 Nishi-Ikebukuro, Toshima-ku, Tokyo 171-8501, Japan.
}

\textit{E-mail address}: \texttt{suzuyu@rikkyo.ac.jp}
}
\bigskip

{\small
\textsc{%
Wataru Takeda\\
Department of Applied Mathematics, Tokyo University of Science,\\
1-3 Kagurazaka, Shinjuku-ku, Tokyo 162-8601, Japan.
}

\textit{E-mail address}: \texttt{w.takeda@rs.tus.ac.jp}
}
\bigskip

{\small
\textsc{%
Yuuya Yoshida\\
Nagoya Institute of Technology,\\
Gokiso-cho, Showa-ku, Nagoya 466-8555, Japan.
}

\textit{E-mail address}: \texttt{yyoshida9130@gmail.com}
}
\end{flushleft}
\end{document}